\documentclass[lettersize,journal]{IEEEtran}

\usepackage{amsmath,amssymb,amsfonts,mathtools}
\usepackage{amsthm}
\usepackage{bm}
\newtheorem{theorem}{Theorem}

\newtheorem{example}{Example}
\theoremstyle{remark}
\newtheorem{remark}{\textbf{Remark}}

\usepackage{algorithm}
\usepackage{algpseudocode}

\usepackage{xcolor}     
\usepackage{multicol}   

\usepackage{graphicx}
\graphicspath{{Pictures/}}
\usepackage[caption=false,font=footnotesize,labelfont=sf,textfont=sf]{subfig}
\usepackage{stfloats}

\usepackage{url}
\usepackage{hyperref}

\usepackage{booktabs}
\usepackage{tabularx}
\usepackage{multirow}
\usepackage{makecell}
\usepackage{diagbox}
\usepackage{array}

\usepackage{colortbl}
\usepackage{adjustbox}
\usepackage{comment}
\usepackage{float}

\usepackage{cite}

\begin{document}

\title{Accelerating Regularized Attention Kernel Regression for Spectrum Cartography}

\author{Liping Tao and Chee Wei Tan, \textit{Senior Member, IEEE}

\thanks{Liping Tao (liping.tao@ntu.edu.sg) is with the College of Computing and Data Science, Nanyang Technological University, Singapore 639798.}

\thanks{Chee Wei Tan (cheewei.tan@ntu.edu.sg) is with the College of Computing and Data Science, Nanyang Technological University, Singapore 639798.}
}

\markboth{Journal of \LaTeX\ Class Files,~Vol.~14, No.~8, August~2021}%
{Shell \MakeLowercase{\textit{et al.}}: A Sample Article Using IEEEtran.cls for IEEE Journals}

\maketitle

\begin{abstract}
Spectrum cartography reconstructs spatial radio fields from sparse and heterogeneous wireless measurements, underpinning many sensing and optimization tasks in wireless networks. Attention mechanisms have recently enabled adaptive measurement aggregation via attention kernel-based formulations. However, the resulting exponential kernels exhibit severe spectral imbalance, inducing large condition numbers that render standard iterative solvers ineffective for regularized attention kernel regression. This paper proposes a Learning-based Attention Kernel Regression (LAKER) algorithm for accelerating regularized attention kernel regression in spectrum cartography. The key idea is to learn a data-dependent preconditioner that captures the inverse spectral structure of the attention kernel system, directly reducing the condition number bottleneck. The preconditioner is obtained by solving a regularized maximum-likelihood estimation problem via a shrinkage-regularized convex--concave procedure, and is integrated with a preconditioned conjugate gradient solver for efficient optimization, whose solution is used for radio map reconstruction. Extensive experiments demonstrate that LAKER significantly reduces condition numbers by up to three orders of magnitude, accelerates convergence by over twenty-fold compared to baselines, and maintains high reconstruction accuracy, establishing learning-based preconditioning as an effective approach for attention kernel regression in spectrum cartography.
\end{abstract}


\section{Introduction}
Spectrum cartography offers a principled framework for constructing spatially continuous radio environment representations from sparse measurements \cite{romero2022radio}. By estimating radio quantities across space, frequency, and time, it enables a broad range of applications including spectrum monitoring, interference management, localization, and resource allocation \cite{shrestha2023spectrum}. As wireless networks evolve toward Sixth-Generation (6G) systems with global coverage, integrated sensing, and AI-native designs, exemplified by simulation platforms such as NVIDIA Sionna \cite{chander2024sionna}, spectrum cartography is increasingly recognized as a measurement-driven inference problem involving sparse, heterogeneous observations rather than simple interpolation \cite{chen2025dynamic}. A central task within this framework is radio map reconstruction \cite{levie2021radiounet, romero2022radio}, which recovers spatial radio fields such as received signal strength, path loss, and power spectral density, from limited measurements \cite{shrestha2023spectrum, romero2024theoretical}. These maps serve as the foundational representation for downstream sensing, planning, and optimization, underscoring the need for accurate and computationally efficient reconstruction methods \cite{chen2025dynamic}.

A large body of work has investigated radio map reconstruction using model-based, interpolation-based, and learning-based approaches \cite{shrestha2023spectrum, romero2017learning}. Model-based methods rely on propagation assumptions, while interpolation-based techniques such as kernel smoothing estimate the field from nearby observations. Both categories often degrade under sparse measurements, model mismatch, or complex environments \cite{shrestha2023spectrum}. Learning-based approaches instead aim to infer the spatial field directly from data, for example via neural networks that map locations or environmental features to radio signals \cite{teganya2019location, jaensch2026radiomap, jaensch2026radio}. Among these, kernel-based learning methods stand out for their principled formulation, casting spectrum cartography as regularized regression in reproducing kernel Hilbert spaces \cite{romero2015stochastic, romero2017learning}. Specifically, Fig.~\ref{fig:sc_kernel_intro} illustrates this pipeline. Sparse received signal strength measurements in Fig.~\ref{fig:sc_kernel_intro} (a), simulated using NVIDIA Sionna RT over an urban scene, are used to reconstruct the radio map in Fig.~\ref{fig:sc_kernel_intro} (b). The radio field is modeled as $f(\mathbf{x}) = \sum_{i=1}^{n} \alpha_i k(\mathbf{x}, \mathbf{x}_i)$, where $k(\cdot,\cdot)$ captures spatial similarity between locations, as illustrated in Fig.~\ref{fig:sc_kernel_intro} (c). This formulation leads to the standard regularized kernel learning problem \cite{romero2022radio, teganya2019location, romero2017learning} (cf. Problem \eqref{prob:sc-primal}).

\begin{figure*}[t]
    \centering
    \includegraphics[width=\linewidth]{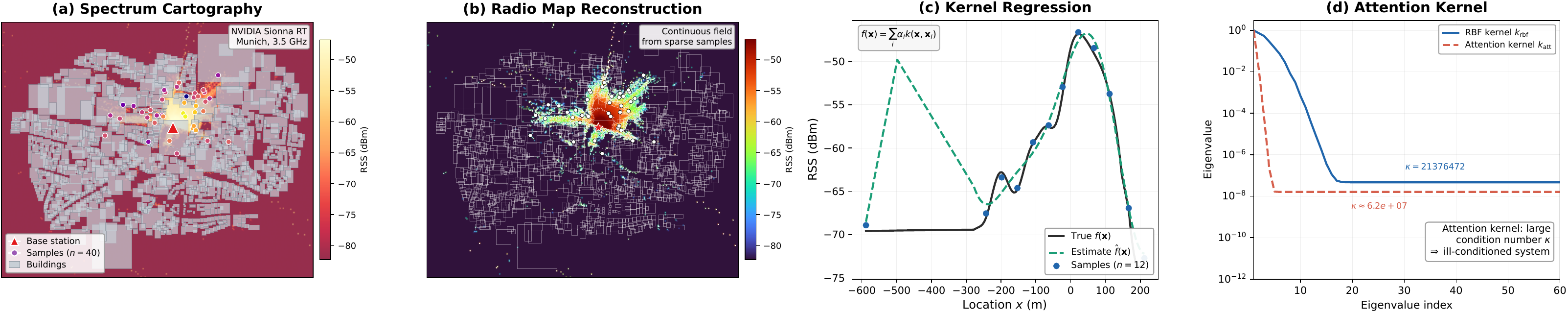}
    \caption{Spectrum cartography and kernel-based radio map reconstruction, simulated using \textit{NVIDIA Sionna RT}~\cite{chander2024sionna} over the Munich urban scene at $f_c=3.5$~GHz.
    \textbf{(a)}~Sparse received signal strength measurements at $n=40$ sensor locations.
    \textbf{(b)}~Reconstructed continuous radio map from sparse observations.
    \textbf{(c)}~Kernel regression view: the radio field $f(\mathbf{x})=\sum_{i=1}^{n}\alpha_i k(\mathbf{x},\mathbf{x}_i)$ estimated from $n=12$ samples along a cross-section, where $k(\mathbf{x}_i,\mathbf{x}_j)$ captures spatial similarity.
    \textbf{(d)}~Eigenvalue spectra of a Radial Basis Function (RBF) kernel $k_{\mathrm{rbf}}$, a standard stationary kernel with smoothly decaying spectrum, and an attention-induced kernel $k_{\mathrm{att}}$: the exponential attention transformation yields a severely imbalanced spectrum with $\kappa \gg 1$.}
    \label{fig:sc_kernel_intro}
\end{figure*}

More recently, attention mechanisms \cite{vaswani2017attention} have emerged as a data-adaptive framework for aggregating heterogeneous measurements, where interactions are governed by learned embeddings rather than fixed kernels \cite{viet2025spatial}. Representative examples include transformer-based spectrum cartography models such as STORM \cite{viet2025spatial} and related attention-driven approaches \cite{chen2024radio, tian2021transformer, liaq2026visual, pandey2021joint}, which demonstrate the effectiveness of embedding-based adaptive attention weighting. This naturally leads to attention kernels by replacing $k(\cdot,\cdot)$ with an attention-induced similarity measure, yielding attention kernel regression as an extension of kernel-based learning. Despite their expressive power, attention-induced kernels introduce significant computational challenges. In scaled dot-product attention, pairwise affinities are computed as inner products in a learned embedding space followed by exponential or softmax normalization \cite{vaswani2017attention}, which amplifies similarities and typically produces highly imbalanced spectra with a few dominant eigenvalues and many near-zero ones. This results in large condition numbers and numerical instability \cite{song2024solving}, as illustrated in Fig.~\ref{fig:sc_kernel_intro} (d). More broadly, the behavior of attention kernel regression is governed by spectral properties: learning progresses unevenly across eigenmodes, and components associated with small eigenvalues are harder to recover \cite{bordelon2020spectrum}. This effect depends on spectral decay, regularization, feature structure \cite{cheng2024comprehensive}, and alignment with the target function \cite{gavrilopoulos2025geometrical}, ultimately leading to severely ill-conditioned systems that challenge accurate and scalable optimization.

To reduce the computational cost of large-scale kernel regression problems, particularly those involving complex kernels such as attention-induced ones, prior work has developed various approximation and preconditioning techniques. One line of work employs low-rank approximations, such as Nystr\"om methods and random feature mappings, to replace the kernel matrix and reduce complexity \cite{drineas2005nystrom, rahimi2007random}. Another focuses on preconditioned iterative solvers, leveraging sketching or structured decompositions to accelerate conjugate-gradient-type methods \cite{avron2017faster,shabat2021fast}. More recent studies further explore data-aware and robust preconditioning strategies to improve stability under heterogeneous conditions \cite{diaz2023robust}. However, these approaches are primarily developed for classical kernels and standard kernel ridge regression. They typically rely on fixed kernel structures, low-rank approximations, or favorable spectral decay, and are not directly tailored to the non-stationary, data-adaptive kernels encountered in spectrum cartography, such as learnable attention-induced kernels, nor do they explicitly exploit the structure of attention kernel regression.

\vspace{1mm}
\textit{Contributions.}
To address the above challenge, we study attention kernel regression in spectrum cartography and propose a learning-based method that adapts to the inverse spectral structure of the attention kernel system to directly mitigate the condition number bottleneck induced by exponential attention kernels. The main contributions are as follows:
\begin{itemize}
    \item We formulate regularized attention kernel regression as a core computational problem for kernel learning-based radio map reconstruction, and characterize the severe spectral ill-conditioning induced by attention kernels.

    \item We develop the Learning-based Attention Kernel Regression (LAKER) algorithm, which learns a Difference-of-Convex (DC) programming--based preconditioner for the regularized attention kernel regression problem. LAKER constructs a regularized Maximum-Likelihood Estimation (MLE) formulation solved via the Convex--Concave Procedure (CCCP) with shrinkage, and integrates it with a Preconditioned Conjugate Gradient (PCG) solver, whose solution is directly used for radio map reconstruction.

    \item Extensive numerical example experiments on spectrum cartography tasks demonstrate that LAKER significantly improves spectral conditioning and solver efficiency across problem sizes. Against first-order methods, preconditioner methods, and convex solvers, LAKER achieves faster convergence and substantial speedup while maintaining high reconstruction accuracy. The source code is publicly available.\footnote{Source code is available at \url{https://github.com/convexsoft/kernelSC}.}
\end{itemize}

\section{Related Work}
\label{sec:rw}
\subsection{Kernel Regression Learning in Spectrum Cartography}
A fundamental formulation of spectrum cartography arises from kernel-based learning \cite{romero2022radio, scholkopf2002learning}, where the radio field is modeled as a function in a reproducing kernel Hilbert space. Given sparse measurements, estimation is cast as a regularized empirical risk minimization problem, whose solution admits a finite kernel expansion over training samples via the representer theorem \cite{scholkopf2001generalized}. The infinite-dimensional estimation problem thus reduces to solving a regularized linear system defined by a kernel matrix, exposing a common operator-inversion structure shared by a broad class of spectrum cartography methods. This formulation extends naturally to more structured settings. In semiparametric reproducing kernel Hilbert space regression \cite{romero2015stochastic}, the radio field is decomposed into parametric and nonparametric components to incorporate prior knowledge while preserving flexibility. Location-free spectrum cartography \cite{teganya2019location} replaces spatial coordinates with signal-derived features while retaining the same kernel regression structure. Kernel-based methods have further been generalized to practical scenarios such as learning power spectral density maps from quantized measurements \cite{romero2017learning}, employing vector-valued reproducing kernel Hilbert space formulations and support-vector-regression-type estimators. Despite these variations, the resulting solutions consistently reduce to solving regularized linear systems, revealing a unified computational structure underlying many spectrum cartography methods. This perspective exposes the central role of the kernel matrix in jointly determining statistical performance and computational complexity.

\subsection{Transformer Attention in Spectrum Cartography}
Transformer attention-based approaches have recently emerged as a powerful paradigm for spectrum cartography, where attention mechanisms play diverse functional roles beyond conventional representation learning. In particular, each spatial measurement or grid feature is first mapped to a latent embedding vector, which is further projected into query, key, and value representations to compute attention weights based on pairwise similarity \cite{vaswani2017attention, su2024roformer}. This embedding-based formulation induces a data-dependent similarity structure over the input samples, effectively defining an adaptive kernel in the learned feature space. Early works such as RadioTrans \cite{tian2021transformer} and subsequent CNN–Transformer hybrids \cite{chen2024radio, li2025rmtransformer} primarily employ self-attention to capture long-range spatial dependencies and enhance global context modeling in grid-based radio map reconstruction. Extending this line, STORM \cite{viet2025spatial} adopts a point-based formulation, where attention acts as a permutation-invariant measurement fusion operator, effectively serving as a learned, input-dependent spatial kernel for adaptive estimation. Beyond reconstruction, attention has also been utilized for higher-level interaction and decision-making tasks; for example, TxSTrans \cite{liaq2026visual} integrates self- and cross-attention to model environment–candidate interactions and enables efficient base station site selection. Furthermore, generative formulations have been explored, where attention functions as a spatial interpolation operator: \cite{pandey2021joint} employs self-attention to aggregate neighboring received signal strength samples and synthesize realistic radio fingerprints at unseen locations. These works collectively demonstrate that transformer attention in spectrum cartography can act as a unifying operator for estimation, representation, fusion, interpolation, and decision-making across different problem settings.

\subsection{Kernel Ridge Regression}
Kernel Ridge Regression (KRR) combines kernel methods with $\ell_2$ regularization and provides a fundamental framework for nonlinear function estimation. Recent advances show that its generalization behavior such as bias--variance trade-offs, benign overfitting, and multiple-descent is governed by the spectral properties of the kernel operator \cite{gavrilopoulos2025geometrical, cheng2024comprehensive}. On the computational side, scalable techniques, including divide-and-conquer and random-feature-based preconditioning, significantly improve efficiency while maintaining accuracy \cite{avron2017faster}. KRR has also been extended to applications such as dynamical system forecasting and structured learning \cite{bollt2020regularized}. However, attention-induced kernels introduce non-stationary and data-adaptive structures that challenge existing methods. This work addresses this issue by developing a learning-based preconditioning framework for attention kernel regression.

\section{Regularized Attention Kernel Regression}
\label{sec:ker}
Spectrum cartography aims to reconstruct a spatial radio field $r(\mathbf{x})$ from sparse measurements $\{(\mathbf{x}_i, y_i)\}_{i=1}^n$, where $\mathbf{x}_i \in \mathbb{R}^{d_x}$ denotes the measurement location and $y_i$ denotes the noisy observation at $\mathbf{x}_i$. A principled approach represents $r(\mathbf{x})$ by a function $f \in \mathcal{H}_k$ in a reproducing kernel Hilbert space $\mathcal{H}_k$, leading to the regularized minimization problem \cite{romero2022radio, teganya2019location, romero2017learning}:
\begin{equation}
\label{prob:sc-primal}
\min_{f \in \mathcal{H}_k}
\;
\sum_{i=1}^n \bigl(y_i - f(\mathbf{x}_i)\bigr)^2
\;+\;
\lambda \|f\|_{\mathcal{H}_k}^2,
\qquad \lambda > 0,
\end{equation}
which is also known as the kernel ridge regression problem \cite{scholkopf2002learning}. By the representer theorem \cite{scholkopf2001generalized}, its solution admits a finite expansion over the training samples:
\begin{equation}
f(\mathbf{x}) = \sum_{i=1}^n \alpha_i\, k(\mathbf{x}_i, \mathbf{x}),
\end{equation}
where $k(\cdot,\cdot)$ is a positive semidefinite kernel function. Substituting into \eqref{prob:sc-primal} yields the finite-dimensional problem:
\begin{equation}
\label{prob:sc-objective}
\min_{\boldsymbol{\alpha} \in \mathbb{R}^n}
\;
\|\mathbf{y} - K\boldsymbol{\alpha}\|_2^2
\;+\;
\lambda \boldsymbol{\alpha}^\top K \boldsymbol{\alpha},
\end{equation}
whose first-order optimality condition gives the regularized linear
system:
\begin{equation}
\label{eq:sc-krr}
(K + \lambda I)\boldsymbol{\alpha} = \mathbf{y},
\end{equation}
where $K + \lambda I$ is symmetric positive definite for any $\lambda > 0$, since $K$ is positive semidefinite.

\begin{figure}[t]
    \centering
    \includegraphics[width=1.0\linewidth]{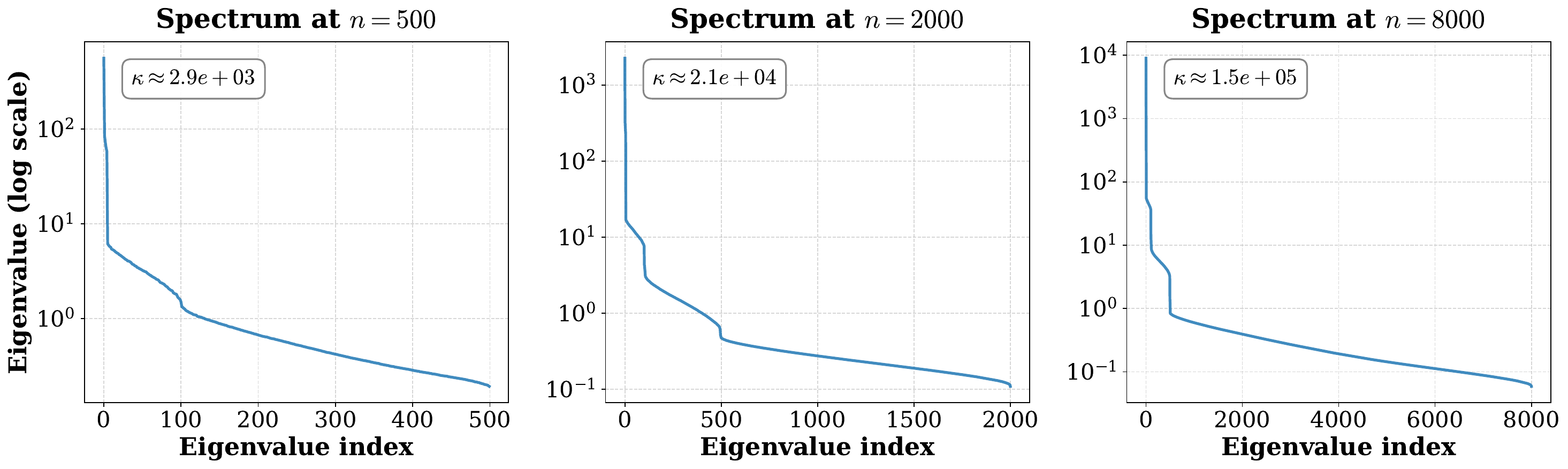}
    \caption{Spectrum of $(\lambda I + G)$ as $n$ increases from
    $500$ to $8000$ ($\lambda = 0.01$). The rapidly growing
    condition number $\kappa$ confirms that attention kernel
    regression remains increasingly ill-conditioned with scale,
    even after regularization.}
    \label{fig:exam-attention-kernel-spectrum}
\end{figure}

\subsection{Regularized Attention Kernel Regression}
In modern data-driven spectrum cartography, each measurement location $\mathbf{x}_i \in \mathbb{R}^{d_x}$ is associated with a learned feature embedding $\mathbf{e}_i \in \mathbb{R}^{d_e}$, obtained via a parametric mapping such as neural encoders or attention-based feature extraction. For example, transformer-based estimators such as STORM \cite{viet2025spatial} compute predictions via attention-weighted aggregation, where weights are defined through exponential similarities in a learned embedding space, effectively inducing a data-dependent similarity measure. Collecting the embeddings into the matrix $E = [\mathbf{e}_1^\top; \dots;\mathbf{e}_n^\top] \in \mathbb{R}^{n \times {d_e}}$, this perspective naturally leads to an exponential attention kernel \cite{song2024solving, viet2025spatial}:
\begin{equation}
\label{eq:sc-exp-kernel}
G_{ij} = \exp(\langle \mathbf{e}_i, \mathbf{e}_j \rangle),
\quad
G = \exp(EE^\top),
\end{equation}
where the exponential is applied element-wise. 

Replacing $K$ with $G$ in \eqref{prob:sc-objective} directly yields the regularized attention kernel regression problem:
\begin{equation}
\label{prob:sc-attn-kernel-regression}
\min_{\boldsymbol{\alpha} \in \mathbb{R}^n}
\;
\|G\boldsymbol{\alpha} - \mathbf{y}\|_2^2
\;+\;
\lambda \boldsymbol{\alpha}^\top G \boldsymbol{\alpha},
\end{equation}
where $\mathbf{y} \in \mathbb{R}^n$ collects the sparse noisy measurements. Its first-order optimality condition yields:
\begin{equation}
\label{eq:sc-final}
(G + \lambda I)\boldsymbol{\alpha} = \mathbf{y},
\end{equation}
where $G + \lambda I$ is symmetric positive definite for any $\lambda > 0$, and $\boldsymbol{\alpha}$ are the kernel expansion coefficients at the training locations. Once $\boldsymbol{\alpha}$ is obtained by solving \eqref{eq:sc-final}, the radio field is reconstructed at any query location $\mathbf{x}$ via:
\begin{align}
\label{eq:sc-radio-map-reconstruction}
\hat{r}(\mathbf{x})
=
\sum_{i=1}^{n}
G(\mathbf{x}, \mathbf{x}_i)\,\alpha_i,
\end{align}
where $G(\mathbf{x}, \mathbf{x}_i) = \exp(\langle \mathbf{e}(\mathbf{x}), \mathbf{e}_i \rangle)$ denotes the attention kernel evaluated between the query location $\mathbf{x}$ and training location $\mathbf{x}_i$, with $\mathbf{e}(\mathbf{x})$ the embedding of $\mathbf{x}$.

\subsection{Spectral Ill-Conditioning of Attention Kernels}
\label{sec:spec-discussion}
Regularized attention kernel regression leads to the linear system $(\lambda I + G)\boldsymbol{\alpha} = \mathbf{y}$, whose numerical behavior is governed by the spectrum of $G$. While adding $\lambda I$ ensures positive definiteness, it does not reduce the spectral imbalance induced by the exponential kernel. To illustrate this, consider an idealized clustered setting where $n$ samples form $Q$ well-separated groups of equal size $q = n/Q$. In this case, $G$ is approximately block-diagonal, with each block approximately rank-one. Consequently, each cluster contributes a single dominant eigenvalue of order $\mathcal{O}(n/Q)$, while the remaining eigenvalues are near zero before regularization and near $\lambda$ after. The condition number of the regularized system scales as:
\begin{equation}
\kappa(\lambda I + G)
\approx 1 + \frac{n}{Q\lambda},
\end{equation}
which grows linearly with the problem size $n$. Thus, even for fixed $\lambda > 0$, the system remains severely ill-conditioned as $n$ increases. This spectral imbalance significantly affects iterative solvers: components associated with dominant eigenvalues converge rapidly, whereas those associated with small eigenvalues converge slowly, leading to stagnation. As shown in Fig.~\ref{fig:exam-attention-kernel-spectrum}, the condition number grows rapidly with $n$, confirming that attention kernel regression remains increasingly ill-conditioned with scale.

\begin{figure*}[t]
    \centering
    \includegraphics[width=\linewidth]{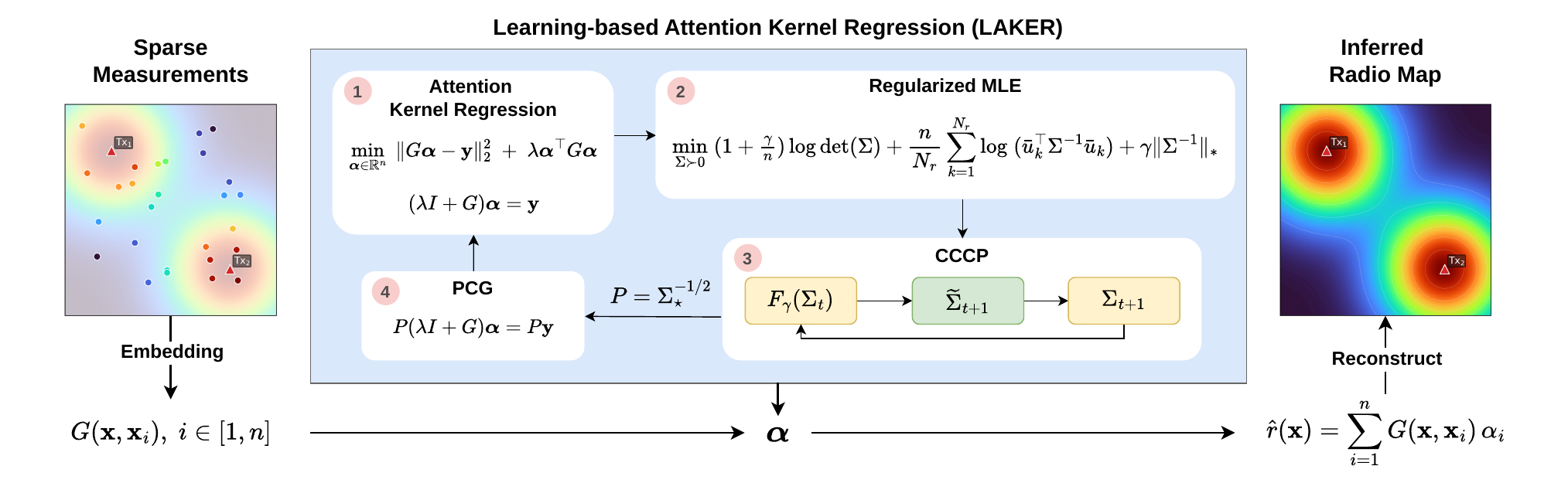}
    \caption{Overview of the Learning-based Attention Kernel Regression (LAKER) algorithm for radio map reconstruction in spectrum cartography. Sparse measurements are encoded via attention kernels to form the regression system. A regularized Maximum Likelihood Estimation (MLE) problem with a Difference-of-Convex (DC) structure is solved via Convex--Concave Procedure (CCCP) to learn a data-dependent preconditioner. The learned preconditioner is then used in a Preconditioned Conjugate Gradient (PCG) method to recover the coefficients, which are finally used for radio map reconstruction.}
    \label{fig:kernel_algorithm}
\end{figure*}

\vspace{1mm}
\textit{Preconditioning Objective.}
The ill-conditioning of attention kernel regression is intrinsic to the exponential kernel: the similarity amplification creates a few dominant eigenmodes and many near-zero ones, leading to severe eigenvalue imbalance even after adding $\lambda I$. To solve $(\lambda I + G)\boldsymbol{\alpha} = \mathbf{y}$ efficiently, a preconditioner $P$ is needed to reduce the condition number of the preconditioned system. The preconditioning objective is:
\begin{equation}
\label{eq:attention-kernel-objective}
\min_{P}\;\kappa\!\bigl(P(\lambda I + G)\bigr)
=
\frac{\lambda_{\max}\!\bigl(P(\lambda I + G)\bigr)}
     {\lambda_{\min}\!\bigl(P(\lambda I + G)\bigr)}.
\end{equation}

\section{Learning-Based Attention Kernel Regression}
This section presents the Learning-based Attention Kernel Regression (LAKER) algorithm for solving regularized attention kernel regression in spectrum cartography, with the overall procedure illustrated in Fig.~\ref{fig:kernel_algorithm} and Algorithm~\ref{alg:LAKER}. LAKER proceeds in three stages. In the first stage, the exponential attention kernel matrix $G$ is constructed from learned feature embeddings via $G_{ij} = \exp(\langle \mathbf{e}_i, \mathbf{e}_j \rangle)$. In the second stage, motivated by Tyler's estimator \cite{sun2014regularized, soloveychik2015tyler, palomar2025portfolio}, LAKER formulates a regularized MLE problem whose objective admits a DC structure \cite{tao1988duality, shen2016disciplined}. This DC problem is solved via the CCCP iteration \cite{lipp2016variations, yuille2003concave} to learn a symmetric positive definite matrix $\Sigma_{\star} \in \mathbb{R}^{n \times n}$, from which the preconditioner is constructed as $P = \Sigma_{\star}^{-1/2}$. The matrix $P$ is designed so that $P(\lambda I + G) \approx I$, directly reducing the condition number of the preconditioned system. The learned $P$ is used within the PCG method to efficiently solve:
\begin{align}
P(\lambda I + G)\boldsymbol{\alpha} = P\mathbf{y}.
\end{align}
In the final stage, the recovered coefficients $\boldsymbol{\alpha}$ are finally used to reconstruct the radio map via \eqref{eq:sc-radio-map-reconstruction}.

\subsection{Tyler's Estimator}
Tyler's estimator \cite{tyler1987statistical} is derived from the Angular Central Gaussian (ACG) model, which represents observations as scaled Gaussian vectors $s_i = \sqrt{\tau_i}\,\xi_i$, where $\xi_i \sim \mathcal{N}(0,\Sigma)$ and $\tau_i > 0$. The corresponding ACG negative log-likelihood is:
\begin{equation}
\ell(\Sigma) = \log\det(\Sigma)
+ \frac{n}{N_r} \sum_{i=1}^{N_r} \log(s_i^\top \Sigma^{-1} s_i).
\end{equation}

Note that $\ell(\Sigma)$ is scale-invariant, satisfying $\ell(c\Sigma) = \ell(\Sigma)$ for any $c > 0$. Setting its gradient to zero yields the fixed-point equation \cite{sun2014regularized}:
\begin{equation}
\Sigma = \frac{n}{N_r} \sum_{i=1}^{N_r}
\frac{s_i s_i^\top}{s_i^\top \Sigma^{-1} s_i}.
\label{eq:tyler-fixed-point}
\end{equation}
Since the fixed-point equation determines $\Sigma$ only up to a positive scalar, the constraint $\operatorname{tr}(\Sigma) = n$ is imposed to ensure a unique solution.

\subsection{Regularized Maximum-Likelihood Estimation}
Given the attention kernel operator $\lambda I+G\succ0$ and $(\lambda I + G) \in \mathbb{R}^{n \times n}$, for $k=\{1,\dots,N_r$\}, we generate random directions by:
\begin{align}
\label{eq:random_directions}
z_k \sim \mathcal{N}(0,I_n), 
\quad
u_k = (\lambda I + G) z_k,
\quad
\bar u_k = \frac{u_k}{\|u_k\|_2},
\end{align}
where $k$ indexes the random directions.

Since $u_k \sim \mathcal{N}(0,(\lambda I + G)^2)$, the normalized vectors $\bar u_k$ follow an angular distribution whose shape is governed by $(\lambda I + G)^2$. Based on these samples, we estimate a positive definite matrix $\Sigma$ via the regularized MLE problem:
\begin{align}
\min_{\Sigma\succ0}\;
\bigl(1+\tfrac{\gamma}{n}\bigr)\log\det(\Sigma)
+
\frac{n}{N_r}\sum_{k=1}^{N_r}
\log\!\bigl(\bar u_k^\top \Sigma^{-1}\bar u_k\bigr)
+
\gamma \|\Sigma^{-1}\|_*,
\label{prob:kernel-mle}
\end{align}
where $\|\cdot\|_*$ denotes the nuclear norm. Since $\Sigma^{-1}\succ0$, there exists $\|\Sigma^{-1}\|_*=\mathrm{tr}(\Sigma^{-1})$.

\begin{remark}
Since $\Sigma \succ 0$, the inverse $\Sigma^{-1}$ is symmetric positive definite. Therefore, its singular values are exactly its positive eigenvalues, and one has:
\begin{align}
\|\Sigma^{-1}\|_* = &\sum_{i=1}^n \sigma_i(\Sigma^{-1})
= \sum_{i=1}^n \sqrt{\lambda_i((\Sigma^{-1})^{\top}(\Sigma^{-1}))} \notag \\
= &\sum_{i=1}^n \lambda_i(\Sigma^{-1})
= \mathrm{tr}(\Sigma^{-1}).
\end{align}
Thus, $\|\Sigma^{-1}\|_*=\mathrm{tr}(\Sigma^{-1})$.
\end{remark}

Obviously, the estimated matrix $\Sigma$ captures the spectral structure of $(\lambda I + G)$ and satisfies $\Sigma \propto (\lambda I + G)^2$. This implies that $\Sigma$ shares the same eigenvectors as $(\lambda I + G)$, while its eigenvalues scale quadratically with those of $(\lambda I + G)$. Thus, we construct the preconditioner as:
\begin{align}
\label{eq:preconditioner-Sigmma}
P = \Sigma_{\star}^{-1/2},
\end{align}
where $\Sigma_{\star}$ is the solution of \eqref{prob:kernel-mle}. \eqref{eq:preconditioner-Sigmma} effectively approximates $(\lambda I + G)^{-1}$ up to a scaling factor. Since such scaling does not affect the condition number, the resulting preconditioner is spectrally aligned with the inverse operator and significantly improves the condition number of $(\lambda I+G)\boldsymbol{\alpha}=\mathbf{y}$.

\begin{algorithm}[t]
\caption{Learning-based Attention Kernel Regression (LAKER)}
\label{alg:LAKER}
\begin{algorithmic}[1]
\Require $\{\mathbf{x}_i\}_{i=1}^n$, $\mathbf{y}$, $\lambda$, $N_r$, $\gamma$.
\Ensure $\boldsymbol{\alpha}$ and $\hat{r}(\mathbf{x})$.

\State \textbf{Kernel Construction}
\State Construct embeddings $\{\mathbf{e}_i\}_{i=1}^n$ from $\{\mathbf{x}_i\}_{i=1}^n$;
\State Compute $G_{ij} = \exp(\langle \mathbf{e}_i, \mathbf{e}_j \rangle)$;

\vspace{1mm}
\State \textbf{Attention Kernel Regression Learning}
\State Sample $z_k \sim \mathcal{N}(0,I_n)$ for $k=1,\dots,N_r$;
\State $u_k \leftarrow (\lambda I + G) z_k$, \quad $\bar u_k \leftarrow u_k / \|u_k\|_2$;
\State Initialize $\Sigma_0 \leftarrow I$;
\For{$t = 0,1,\dots$ until convergence}
    \State Compute $F_\gamma(\Sigma_t)$ via \eqref{eq:kernel_shrinkage_cccp_F};
    \State $\widetilde{\Sigma}_{t+1} \leftarrow (1-\rho)F_\gamma(\Sigma_t) + \rho I$;
    \State $\Sigma_{t+1} \leftarrow \widetilde{\Sigma}_{t+1} / \big(\mathrm{tr}(\widetilde{\Sigma}_{t+1})/n\big)$;
\EndFor
\State $P \leftarrow \Sigma_{\star}^{-1/2}$;

\vspace{1mm}
\State \textbf{Preconditioned Conjugate Gradient}
\State Initialize $\boldsymbol{\alpha}^0 \leftarrow \mathbf{0}$,
$\mathbf{r}^0 \leftarrow \mathbf{y}$,
$\boldsymbol{\vartheta}^0 \leftarrow P \mathbf{r}^0$,
$\mathbf{p}^0 \leftarrow \boldsymbol{\vartheta}^0$;
\For{$k = 0,1,\dots$ until convergence}
    \State $\delta^k \leftarrow \frac{(\mathbf{r}^k)^\top \boldsymbol{\vartheta}^k}{(\mathbf{p}^k)^\top (\lambda I + G)\mathbf{p}^k}$;
    \State $\boldsymbol{\alpha}^{k+1} \leftarrow \boldsymbol{\alpha}^k + \delta^k \mathbf{p}^k$;
    \State $\mathbf{r}^{k+1} \leftarrow \mathbf{r}^k - \delta^k (\lambda I + G)\mathbf{p}^k$;
    \If{$\|\mathbf{r}^{k+1}\|_2 / \|\mathbf{y}\|_2 \le \varepsilon_{\text{tol}}$}
        \State \textbf{break}
    \EndIf
    \State $\boldsymbol{\vartheta}^{k+1} \leftarrow P \mathbf{r}^{k+1}$;
    \State $\beta^k \leftarrow \frac{(\mathbf{r}^{k+1})^\top \boldsymbol{\vartheta}^{k+1}}{(\mathbf{r}^k)^\top \boldsymbol{\vartheta}^k}$;
    \State $\mathbf{p}^{k+1} \leftarrow \boldsymbol{\vartheta}^{k+1} + \beta^k \mathbf{p}^k$;
\EndFor

\vspace{1mm}
\State \textbf{Radio Map Reconstruction}
\State Reconstruct $\hat{r}(\mathbf{x})$ via $\hat{r}(\mathbf{x}) = \sum_{i=1}^{n} G(\mathbf{x}, \mathbf{x}_i)\,\alpha_i$;
\State \Return $\boldsymbol{\alpha}$ and $\hat{r}(\mathbf{x})$.
\end{algorithmic}
\end{algorithm}

\subsection{Convex--Concave Procedure Iteration}
Let $\Theta=\Sigma^{-1} \succ 0$, then the problem \eqref{prob:kernel-mle} becomes:
\begin{align}
\min_{\Theta\succ0}\;
-
\bigl(1+\tfrac{\gamma}{n}\bigr)\log\det(\Theta)
+
\frac{n}{N_r}\sum_{k=1}^{N_r}
\log\!\bigl(\bar u_k^\top \Theta \bar u_k\bigr)
+
\gamma \|\Theta\|_*,
\label{eq:kernel-mle-theta}
\end{align}
which is obviously a non-convex problem. We further define:
\begin{align}
l(\Theta) &= -\Bigl(1+\tfrac{\gamma}{n}\Bigr)\log\det(\Theta),\\
g(\Theta) &= 
\frac{n}{N_r}\sum_{k=1}^{N_r}\log\!\bigl(\bar u_k^\top \Theta \bar u_k\bigr)
+ \gamma \operatorname{tr}(\Theta). \label{eq:cccp-g-theta}
\end{align}
Here, $l(\Theta)$ is convex and $g(\Theta)$ is concave. Then~\eqref{eq:kernel-mle-theta} can be written as:
\begin{align}
\label{eq:cccp-L-theta}
\min_{\Theta\succ0}\;
L(\Theta)
=
l(\Theta) + g(\Theta).
\end{align}

Defining $h(\Theta) := -g(\Theta)$, we obtain the DC program \cite{tao1988duality, shen2016disciplined}:
\begin{align}
\label{eq:dc-L-theta}
\min_{\Theta\succ0}\;
L(\Theta) = l(\Theta) - h(\Theta),
\end{align}
where both $l$ and $h$ are convex. By Toland duality \cite{toland1978duality}, $h(\Theta)$ admits:
\begin{align}
h(\Theta)
=
\sup_{\Phi}\;
\bigl\{ \langle \Theta, \Phi \rangle - h^*(\Phi) \bigr\}.
\end{align}
Substituting it into \eqref{eq:dc-L-theta} yields:
\begin{align}
L(\Theta)
=
\inf_{\Phi}\;
\bigl\{
l(\Theta) - \langle \Theta, \Phi \rangle + h^*(\Phi)
\bigr\},
\end{align}
which leads to the associated DC dual problem:
\begin{align}
\min_{\Phi}\;
h^*(\Phi) - l^*(\Phi).
\end{align}

At an optimal solution $\Theta^\star$ of \eqref{eq:dc-L-theta}, the DC optimality condition requires $\partial h(\Theta^\star) \subset \partial l(\Theta^\star)$, which reduces to:
\begin{align}
\label{eq:cccp-toland-stationarity}
\nabla l(\Theta^\star) = \nabla h(\Theta^\star),
\end{align}
when $l$ and $h$ are differentiable.

The CCCP method \cite{yuille2003concave, lipp2016variations, yuille2001concave} solves \eqref{eq:cccp-L-theta} iteratively by linearizing the concave part $g(\Theta)$ at $\Theta_t$ to achieve \eqref{eq:cccp-toland-stationarity}, yielding the convex subproblem:
\begin{align}
\Theta_{t+1}
\in
\arg\min_{\Theta\succ0}
\;
l(\Theta) + \langle \nabla g(\Theta_t), \Theta \rangle.
\end{align}
Its optimality condition shows:
\begin{align}
\nabla l(\Theta_{t+1}) = -\nabla g(\Theta_t).
\end{align}
The gradients of $l$ and $g$ are given by:
\begin{align}
\nabla l(\Theta)
&=
-\Bigl(1+\tfrac{\gamma}{n}\Bigr)\Theta^{-1},
\\
\nabla g(\Theta)
&=
\frac{n}{N_r}
\sum_{k=1}^{N_r}
\frac{\bar u_k \bar u_k^\top}
{\bar u_k^\top \Theta \bar u_k}
+
\gamma I.
\end{align}
Therefore:
\begin{align}
\Bigl(1+\tfrac{\gamma}{n}\Bigr)\Theta_{t+1}^{-1}
=
\frac{n}{N_r}
\sum_{k=1}^{N_r}
\frac{\bar u_k \bar u_k^\top}
{\bar u_k^\top \Theta_t \bar u_k}
+
\gamma I,
\end{align}
which leads to the closed-form update:
\begin{align}
\Theta_{t+1}^{-1}
=
\frac{1}{1+\gamma/n}
\left(
\frac{n}{N_r}
\sum_{k=1}^{N_r}
\frac{\bar u_k \bar u_k^\top}
{\bar u_k^\top \Theta_t \bar u_k}
+
\gamma I
\right).
\end{align}
Equivalently, in the $\Sigma$-domain we obtain the CCCP iteration:
\begin{align}
\label{eq:atten-kernel-cccp-iteration}
\Sigma_{t+1}
=
\frac{1}{1+\gamma/n}
\left(
\frac{n}{N_r}
\sum_{k=1}^{N_r}
\frac{\bar u_k \bar u_k^\top}
{\bar u_k^\top \Sigma_t^{-1} \bar u_k}
+
\gamma I
\right).
\end{align}

Note that the objective in \eqref{prob:kernel-mle} is no longer scale-invariant due to the nuclear term. As a result, the optimization may exhibit scale drift, leading to uncontrolled growth or shrinkage of the eigenvalues of $\Sigma$. To improve numerical stability, we normalize $\Sigma_{t+1}$ after each update as:
\begin{align}
\Sigma_{t+1}
\leftarrow
\frac{\Sigma_{t+1}}
{\mathrm{tr}(\Sigma_{t+1})/n}.
\end{align}
The resulting preconditioner is defined as $P = \Sigma_{\star}^{-1/2}$.

\begin{remark}
The CCCP update admits an equivalent Frank--Wolfe (FW) interpretation \cite{frank1956algorithm}. Consider the DC objective $\min_{\Theta \succ 0} L(\Theta)=l(\Theta)-h(\Theta)$ and it can form $\min_{\Theta,\hat l}\; \hat l - h(\Theta)\ \text{s.t.}\ l(\Theta)\le \hat l$. Linearizing $h$ at $\Theta_t$ yields the FW subproblem $\min_{\Theta,\hat l}\; \hat l - \langle \nabla h(\Theta_t), \Theta \rangle\ \text{s.t.}\ l(\Theta)\le \hat l$, whose optimality condition is $\nabla l(\Theta_{t+1})=\nabla h(\Theta_t)$. Substituting the gradients recovers exactly \eqref{eq:atten-kernel-cccp-iteration}, showing that the CCCP iteration is equivalent to applying FW to the DC formulation.
\end{remark}

\subsection{Shrinkage-Regularized CCCP Iteration}
To improve numerical robustness when $N_r < n$, especially under limited samples or severe spectral imbalance \cite{wiesel2011unified}, we introduce a stabilized shrinkage-regularized CCCP update inspired by \cite{chen2011robust}. Firstly, we define:
\begin{align}
\label{eq:kernel_shrinkage_cccp_F}
F_{\gamma}(\Sigma_t)
=
\frac{1}{1+\gamma/n}
\left(
\frac{n}{N_r}
\sum_{k=1}^{N_r}
\frac{\bar u_k \bar u_k^\top}
{\bar u_k^\top \Sigma_t^{-1} \bar u_k + \varepsilon}
+
\gamma I
\right),
\end{align}
where $\varepsilon>0$ is a small safeguard parameter to prevent numerical instability. 

Secondly, we apply isotropic shrinkage regularization:
\begin{align}
\label{eq:kernel_shrinkage_cccp_Sigma}
\widetilde{\Sigma}_{t+1}
=
(1-\rho)\,F_{\gamma}(\Sigma_t)
+
\rho I,
\end{align}
where $\rho \in [0,1]$ controls the strength of shrinkage. In practice, $\rho$ is increased when the smallest eigenvalue becomes too small, and set to zero otherwise. Finally, we normalize:
\begin{align}
\label{eq:kernel_shrinkage_cccp_normalization}
\Sigma_{t+1}
=
\frac{\widetilde{\Sigma}_{t+1}}
{\mathrm{tr}(\widetilde{\Sigma}_{t+1})/n}.
\end{align}

\begin{remark}
The method only requires matrix--vector products of the form $u_k = (\lambda I + G) z_k = \lambda z_k + G z_k$, and does not require forming or storing the full kernel matrix $G$. This makes it suitable for large-scale settings, where $(\lambda I + G)$ can be accessed through efficient operator evaluations.
\end{remark}

\begin{theorem}
Let $\varepsilon > 0$, $\rho \in (0,1]$, and $\gamma \ge 0$. Assume that $\sum_{k=1}^{N_r} \bar u_k \bar u_k^\top \;\succ\; 0$. Define:
\begin{align}
&F_{\gamma,\varepsilon}(\Sigma)
=
\frac{1}{1+\gamma/n}
\left(
\frac{n}{N_r}
\sum_{k=1}^{N_r}
\frac{\bar u_k \bar u_k^\top}
{\bar u_k^\top \Sigma^{-1} \bar u_k + \varepsilon}
+
\gamma I
\right), \\
&\widetilde F_{\gamma,\varepsilon,\rho}(\Sigma)
=
(1-\rho)\,F_{\gamma,\varepsilon}(\Sigma) + \rho I, \\
&\mathcal{T}(\Sigma)
=
\frac{n\,\widetilde F_{\gamma,\varepsilon,\rho}(\Sigma)}
{\mathrm{tr}(\widetilde F_{\gamma,\varepsilon,\rho}(\Sigma))}.
\end{align}
Then there exists at least one matrix $\Sigma_\star \succ 0$ such that:
\begin{align}
\mathcal{T}(\Sigma_\star) = \Sigma_\star,
\qquad
\mathrm{tr}(\Sigma_\star)=n.
\end{align}
\end{theorem}

\begin{proof}
Consider the set $\mathcal{S}=\left\{\Sigma \succeq 0 \;\middle|\;\mathrm{tr}(\Sigma)=n\right\}$, which is convex and closed. For any $\Sigma \succ 0$, all denominators are strictly positive, so $F_{\gamma,\varepsilon}(\Sigma)$ is well-defined. Moreover, since $\rho \in (0,1]$:
\begin{align}
\widetilde F_{\gamma,\varepsilon,\rho}(\Sigma)
=
(1-\rho)F_{\gamma,\varepsilon}(\Sigma) + \rho I
\succeq \rho I,
\end{align}
which guarantees $\widetilde F_{\gamma,\varepsilon,\rho}(\Sigma) \succ 0$.

After normalization, we have:
\begin{align}
\mathcal{T}(\Sigma) \succ 0,
\qquad
\mathrm{tr}(\mathcal{T}(\Sigma))=n,
\end{align}
so $\mathcal{T}$ maps $\mathcal{S}$ into itself.

The mapping $\mathcal{T}$ is continuous, and its outputs remain bounded due to the uniform lower bound $\rho I$. Therefore, $\mathcal{T}$ operates on a compact convex subset of $\mathcal{S}$. By Brouwer's fixed-point theorem \cite{granas2003fixed}, there exists at least one $\Sigma_\star \in \mathcal{S}$ such that $\mathcal{T}(\Sigma_\star) = \Sigma_\star$.
\end{proof}

\subsection{Preconditioned Conjugate Gradient}
After obtaining the covariance estimate $\Sigma$ via the shrinkage-regularized CCCP, the preconditioner is constructed as $P = \Sigma_{\star}^{-1/2}$. We aim to solve the linear system:
\begin{align}
(\lambda I + G)\boldsymbol{\alpha} = \mathbf{y},
\end{align}
where $\boldsymbol{\alpha} \in \mathbb{R}^n$ denotes the coefficient vector and $\mathbf{y} \in \mathbb{R}^n$ is the observation vector.

To improve numerical condition number, we consider the left-preconditioned system:
\begin{align}
P(\lambda I + G)\boldsymbol{\alpha} = P\mathbf{y},
\end{align}
where the preconditioner $P$ is designed to reduce the condition number from $\kappa(\lambda I + G)$ to $\kappa(P(\lambda I + G))$.

We adopt the PCG method to iteratively compute the solution. Starting from an initial estimate:
\begin{align}
\boldsymbol{\alpha}^0 = \mathbf{0}, \quad 
\mathbf{r}^0 = \mathbf{y} - (\lambda I + G)\boldsymbol{\alpha}^0, \quad
\boldsymbol{\vartheta}^0 = P \mathbf{r}^0, \quad
\mathbf{p}^0 = \boldsymbol{\vartheta}^0,
\end{align}
the residual and preconditioned residual at iteration $k$ are given by:
\begin{align}
\mathbf{r}^k = \mathbf{y} - (\lambda I + G)\boldsymbol{\alpha}^k, 
\qquad
\boldsymbol{\vartheta}^k = P \mathbf{r}^k.
\end{align}
The step size is computed as:
\begin{align}
\delta^k = \frac{(\mathbf{r}^k)^\top \boldsymbol{\vartheta}^k}{(\mathbf{p}^k)^\top (\lambda I + G)\mathbf{p}^k},
\end{align}
and the solution is updated via:
\begin{align}
\boldsymbol{\alpha}^{k+1} = \boldsymbol{\alpha}^k + \delta^k \mathbf{p}^k.
\end{align}
The residual and the preconditioned residual are updated as:
\begin{align}
\mathbf{r}^{k+1} = \mathbf{r}^k - \delta^k (\lambda I + G)\mathbf{p}^k, \quad
\boldsymbol{\vartheta}^{k+1} = P \mathbf{r}^{k+1}.
\end{align}
The conjugate direction is updated as:
\begin{align}
\beta^k = \frac{(\mathbf{r}^{k+1})^\top \boldsymbol{\vartheta}^{k+1}}{(\mathbf{r}^k)^\top \boldsymbol{\vartheta}^k}, 
\qquad
\mathbf{p}^{k+1} = \boldsymbol{\vartheta}^{k+1} + \beta^k \mathbf{p}^k.
\end{align}

The iteration terminates when the relative residual satisfies:
\begin{align}
{\|\mathbf{r}^k\|_2}/{\|\mathbf{y}\|_2} \le \varepsilon_{\text{tol}}.
\end{align}
or when a target objective accuracy is reached.


\vspace{1mm}
\begin{example}
Consider $n=3$ measurement locations extracted from an NVIDIA Sionna RT simulation of the Munich urban scene at $f_c = 3.5$~GHz \cite{chander2024sionna}. The locations $\mathbf{x}_1, \mathbf{x}_2, \mathbf{x}_3$ yield two-dimensional feature embeddings $\mathbf{e}_i \in \mathbb{R}^2$ and observed received signal strength values:
\begin{align}
\mathbf{e}_1 &= (0.241,\; 0.444), \quad
\mathbf{e}_2 = (-0.336,\; 0.112), \notag \\
\mathbf{e}_3 &= (-0.220,\; 0.353), \notag \\
\mathbf{y} \; &= [-66.14,\; -65.77,\; -77.30]^\top\ \text{dBm}.
\end{align}
The exponential attention kernel $G_{ij} = \exp(\langle \mathbf{e}_i, \mathbf{e}_j \rangle)$ gives:
\begin{align}
G =
\begin{bmatrix}
1.291 & 0.969 & 1.109 \\
0.969 & 1.133 & 1.120 \\
1.109 & 1.120 & 1.189
\end{bmatrix},
\end{align}
and with $\lambda = 0.1$:
\begin{align}
\lambda I + G =
\begin{bmatrix}
1.391 & 0.969 & 1.109 \\
0.969 & 1.233 & 1.120 \\
1.109 & 1.120 & 1.289
\end{bmatrix}.
\end{align}
Solving $(\lambda I + G)\boldsymbol{\alpha} = \mathbf{y}$ yields:
\begin{align}
\boldsymbol{\alpha} = [0.815,\; 5.438,\; -65.406]^\top,
\end{align}
which can be verified by direct substitution. For a query location $\mathbf{x}^\star$ with embedding $\mathbf{e}^\star = (0.051,\; 0.452)$, the kernel values are:
\begin{align}
G(\mathbf{x}^\star, \mathbf{x}_i) = [1.237,\; 1.034,\; 1.160],
\end{align}
and the predicted received signal strength at $\mathbf{x}^\star$ is:
\begin{align}
\hat{r}(\mathbf{x}^\star)
= \sum_{i=1}^{3} G(\mathbf{x}^\star, \mathbf{x}_i)\,\alpha_i
\approx -69.2\ \text{dBm},
\end{align}
close to the Sionna ground truth of $-67.3$~dBm. This example illustrates the end-to-end computation of \eqref{eq:sc-radio-map-reconstruction}: solving $(\lambda I + G)\boldsymbol{\alpha} = \mathbf{y}$ recovers the kernel coefficients $\boldsymbol{\alpha}$, which are then used to predict the radio field at arbitrary locations via kernel expansion. Large-scale evaluations of this procedure are presented in Section~\ref{sec:simu}.
\end{example}

\section{Numerical Example}
\label{sec:simu}

\subsection{Evaluation Setup}
\label{subsec:sim-setup}

\subsubsection{Problem Setup}
We consider a two-dimensional spatial region $\Omega = [0,100]\times[0,100]$~m$^2$, i.e., $d_x=2$. The ground-truth radio field $r(\mathbf{x})$ is generated by superposing multiple transmitters with spatially decaying power profiles and log-normal shadowing. For each realization, $n \in \{50, 200, 500, 1000, 2000\}$ measurement locations $\{\mathbf{x}_i\}_{i=1}^n$ are drawn uniformly from $\Omega$, and noisy observations are collected as:
\begin{align}
y_i = r(\mathbf{x}_i) + \varepsilon_i,
\quad \varepsilon_i \sim \mathcal{N}(0, \sigma_\varepsilon^2),
\quad \sigma_\varepsilon = 1.5, 
\end{align}
where the noise is modeled in the logarithmic (dBm) domain. Each location $\mathbf{x}_i \in \mathbb{R}^{d_x}$ is mapped to a $d_e$-dimensional feature embedding $\mathbf{e}_i \in \mathbb{R}^{d_e}$ with $d_e=10$ via a deterministic position-driven mapping, and the attention kernel is formed as $G_{ij} = \exp(\langle\mathbf{e}_i, \mathbf{e}_j\rangle)$ per \eqref{eq:sc-exp-kernel}. The reconstruction is formulated as the regularized attention kernel regression problem \eqref{prob:sc-attn-kernel-regression} with $\lambda = 10^{-2}$:
\begin{align}
\label{eq:simu-R-definition}
\min_{\boldsymbol{\alpha} \in \mathbb{R}^n} \; R(\boldsymbol{\alpha}) := \|G\boldsymbol{\alpha}-\mathbf{y}\|_2^2 + \lambda \boldsymbol{\alpha}^\top G \boldsymbol{\alpha},
\end{align}
and the radio field at any location is reconstructed via kernel expansion \eqref{eq:sc-radio-map-reconstruction}. Reconstruction performance is evaluated on a dense $45\times45$ grid over $\Omega$.

\begin{table}[t]
\centering
\caption{Key simulation parameters.}
\label{tab:simulation_parameters}
\renewcommand{\arraystretch}{1.0}

\resizebox{\columnwidth}{!}{
\begin{tabular}{l c l}
\toprule
\textbf{Parameter} & \textbf{Value} & \textbf{Description} \\
\midrule

$n$ & $\{50,200,500,1000,2000\}$ & Number of measurements \\
$\Omega$ & $[0,100]\times[0,100]$ m$^2$ & Spatial region \\
$d_x$ & $2$ & Input spatial dimension \\
$\sigma$ & $1.5$ dB & Measurement noise std \\
$d_e$ & $10$ & Embedding dimension \\
$\lambda$ & $10^{-2}$ & Regularization parameter \\
$G$ & $\exp(EE^\top)$ & Attention kernel \\
Grid size & $45\times45$ & Evaluation grid \\

$N_r$ & Adaptive & Number of random directions \\
$\gamma$ & $10^{-1}$ & CCCP regularization \\
$\rho$ & Adaptive & Shrinkage parameter \\

$\eta$ (GD) & Tuned & Gradient descent step size \\
$P_{\text{J}}$ (Jacobi) & $\mathrm{diag}(\lambda I + G)^{-1}$ & Jacobi preconditioner \\
Kernel (GPRT) & Rational quadratic  & Gaussian process kernel type \\
$\sigma^2$ (GPRT) & Tuned & Gaussian process noise variance \\

PCG.TOL & $10^{-10}$--$10^{-11}$ & Residual tolerance \\
TAR.TOL & $10^{-3}$ & Target objective gap \\

\bottomrule
\end{tabular}
}
\end{table}

\vspace{1mm}
\subsubsection{Preconditioner Construction}
As shown in Algorithm LAKER, the preconditioner is learned from random directions as defined in \eqref{eq:random_directions}. The number of random vectors $N_r$ follows a hybrid schedule that scales as $\mathcal{O}(\sqrt{n})$ for small $n$ and linearly for large $n$, balancing estimation quality and computational cost. To improve robustness under undersampling (i.e., when $N_r < n$), an adaptive shrinkage parameter $\rho$ is employed within the CCCP iteration \eqref{eq:kernel_shrinkage_cccp_Sigma}: when $N_r \ge n$, a small fixed regularization $\rho = \epsilon_\rho$ suffices for numerical stability; otherwise, $\rho$ increases with the undersampling ratio $N_r/n$ and is further modulated by $\gamma$. The final preconditioner is given by $P = \Sigma_{\star}^{-1/2}$.

\vspace{1mm}
\subsubsection{Baselines}
We consider baselines at both the numerical layer and the reconstruction layer, with the convex solver CVXPY \cite{diamond2016cvxpy} providing high-accuracy reference solutions.

\vspace{1mm}
\textit{Numerical Layer.}
We consider two representative methods for the linear system $(\lambda I + G)\boldsymbol{\alpha} = \mathbf{y}$:
\begin{itemize}
    \item Gradient Descent (GD) \cite{bottou2018optimization}: a first-order method that iteratively updates $\boldsymbol{\alpha}$ via $\boldsymbol{\alpha}^{(k+1)} = \boldsymbol{\alpha}^{(k)} - \eta \nabla R(\boldsymbol{\alpha}^{(k)})$, where $R(\boldsymbol{\alpha})$ is defined in \eqref{eq:simu-R-definition} and $\eta > 0$ is the step size selected by grid search.

    \item Jacobi Preconditioned Conjugate Gradient
    (Jacobi PCG) \cite{saad2003iterative, cutajar2016preconditioning}: a Krylov subspace method equipped with the diagonal preconditioner $P_{\mathrm{J}} = \mathrm{diag}(\lambda I + G)^{-1}$, which captures only the local diagonal structure of the system.
\end{itemize}

\vspace{1mm}
\textit{Reconstruction Layer.}
We compare against a Gaussian process regression-based radio map reconstruction method (GPRT) \cite{chen2025gprt}, which models spatial correlations via kernel-based probabilistic inference. The prediction at a query location $\mathbf{x}$ takes the form:
\begin{align}
\hat{r}(\mathbf{x}) = k(\mathbf{x}, X)\,(K + \sigma^2 I)^{-1}\mathbf{y},
\end{align}
where $X = [\mathbf{x}_1, \dots, \mathbf{x}_n]^\top$ collects the $n$ training locations, $K \in \mathbb{R}^{n \times n}$ is the kernel matrix with entries $K_{ij} = k(\mathbf{x}_i, \mathbf{x}_j)$, and $k(\mathbf{x}, X) \in \mathbb{R}^{1 \times n}$ denotes the cross-kernel vector at the query point. This shares the same kernel regression structure as \eqref{eq:sc-radio-map-reconstruction}, with $K$ and $k(\mathbf{x}, X)$ playing the roles of $G$ and $G(\mathbf{x}, X)$, respectively. In GPRT, the kernel is a hand-crafted composite incorporating both spatial proximity and terrain (elevation) effects \cite{chen2025gprt}, implemented via an anisotropic Rational Quadratic (RQ) kernel over spatial and elevation dimensions. The two methods thus differ primarily in kernel design: GPRT employs a fixed composite kernel, whereas the algorithm LAKER uses a learned attention kernel.

\vspace{1mm}
\subsubsection{Evaluation Metrics}
Performance is evaluated at both the numerical layer and the reconstruction layer.

\vspace{1mm}
\textit{Numerical Layer.}
We assess solution accuracy and computational efficiency via the following metrics:
\begin{itemize}
    \item Condition number: $\kappa(\lambda I + G)$ for the unpreconditioned system, $\kappa(P(\lambda I + G))$ for the algorithm LAKER, and $\kappa(P_{\mathrm{J}}(\lambda I + G))$ for Jacobi PCG.

    \item Iterations: the number of iterations required to reach a target objective gap of $\text{TAR.TOL} = 10^{-3}$, i.e.,
    \begin{align}
    \label{eq:iterations}
    \frac{|R(\boldsymbol{\alpha}) - R(\boldsymbol{\alpha}_{\mathrm{CVX}})|}{|R(\boldsymbol{\alpha}_{\mathrm{CVX}})|} \le 10^{-3}.
    \end{align}

    \item Residual (Res.): the normalized linear system residual:
    \begin{align}
    \label{eq:residual}
    \text{Res.} = \frac{\|(\lambda I + G)\boldsymbol{\alpha} - \mathbf{y}\|_2}{\|\mathbf{y}\|_2}.
    \end{align}
    
    \item Objective gap (Obj. Gap): the relative gap between the achieved objective value and the convex solver reference:
    \begin{align}
    \label{eq:objective-gap}
    \text{Obj. Gap} = \frac{|R(\boldsymbol{\alpha}) - R(\boldsymbol{\alpha}_{\mathrm{CVX}})|}{|R(\boldsymbol{\alpha}_{\mathrm{CVX}})|}.
    \end{align}  
\end{itemize}

\vspace{1mm}
\textit{Reconstruction Layer.}
Let $\{\mathbf{x}_j\}_{j=1}^M$ denote a dense $45\times45$ evaluation grid over $\Omega$, distinct from the measurement locations $\{\mathbf{x}_i\}_{i=1}^n$. We evaluate reconstruction quality via:
\begin{itemize}
    \item Prediction discrepancy (Pred. Disc.): the relative difference in predicted outputs with respect to the convex solver reference:
    \begin{align}
    \label{eq:prediction-discrepancy}
    \text{Pred. Disc.}
    = \frac{\|G\boldsymbol{\alpha} -
    G\boldsymbol{\alpha}_{\mathrm{CVX}}\|_2}
    {\|G\boldsymbol{\alpha}_{\mathrm{CVX}}\|_2}.
    \end{align}

    \item Root Mean Square Error (RMSE), and Normalized Mean Square Error (NMSE): standard metrics for reconstruction accuracy against the ground-truth radio field:
    \begin{align}
    \text{RMSE} &= \sqrt{\frac{1}{M}\sum_{j=1}^{M}
    \bigl(\hat{r}(\mathbf{x}_j) - r(\mathbf{x}_j)\bigr)^2}, \label{eq:rmse} \\
    \text{NMSE} &= \frac{\sum_{j=1}^{M}
    \bigl(\hat{r}(\mathbf{x}_j) - r(\mathbf{x}_j)\bigr)^2}
    {\sum_{j=1}^{M} r(\mathbf{x}_j)^2}. \label{eq:nmse}
    \end{align}
\end{itemize}

\begin{figure*}[!t]
\centering
\subfloat[\footnotesize Condition number.\label{fig:condition_number_comparison_vs_n}]{
\includegraphics[width=0.24\textwidth]{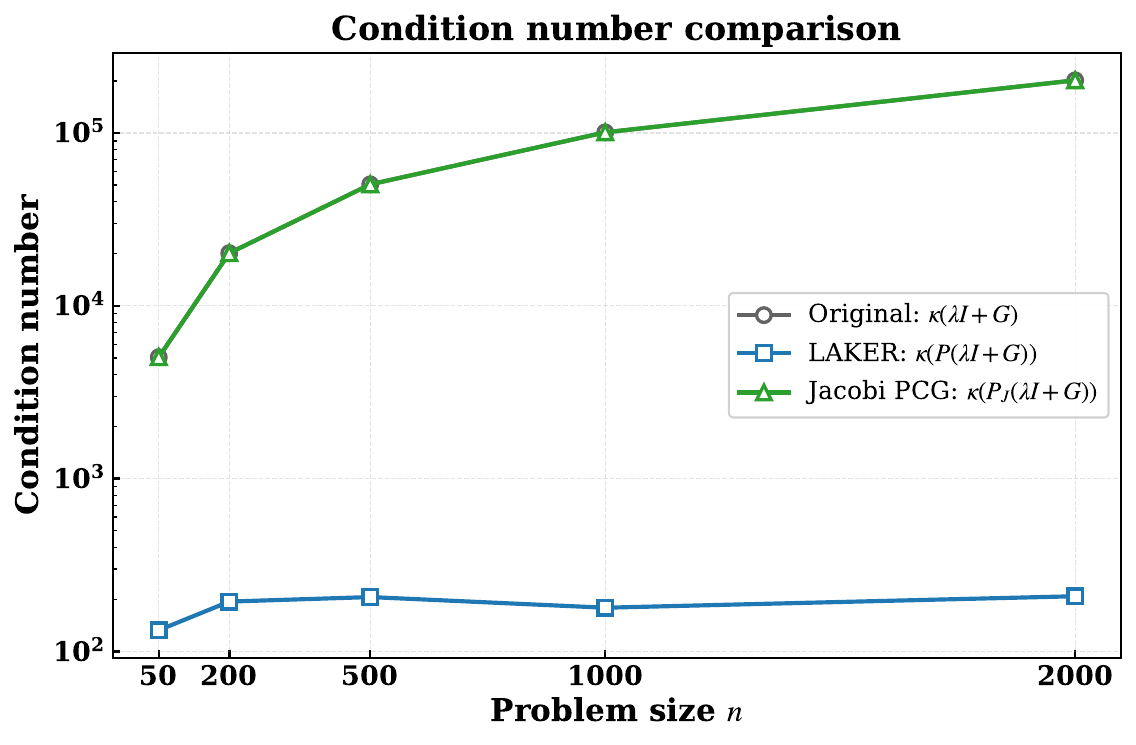}
}
\subfloat[\footnotesize Iterations.\label{fig:iterations_to_accuracy_vs_n}]{
\includegraphics[width=0.24\textwidth]{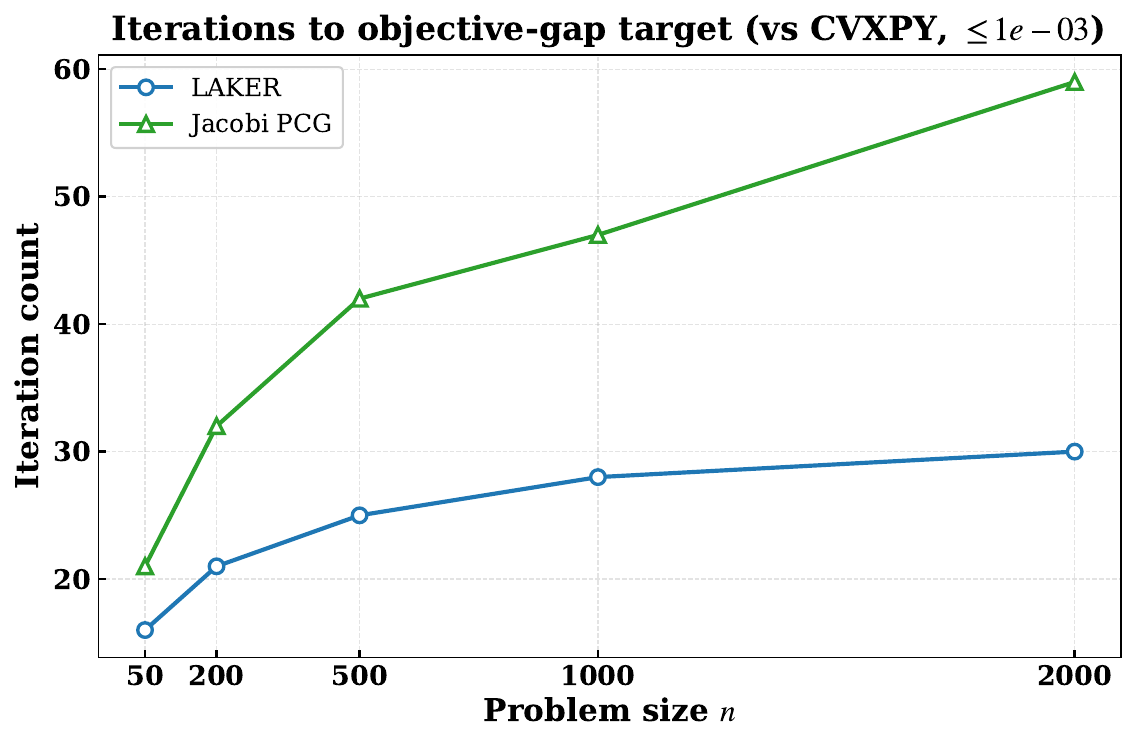}
}
\subfloat[\footnotesize Residual.\label{fig:linear_system_residual_vs_n}]{\includegraphics[width=0.24\textwidth]{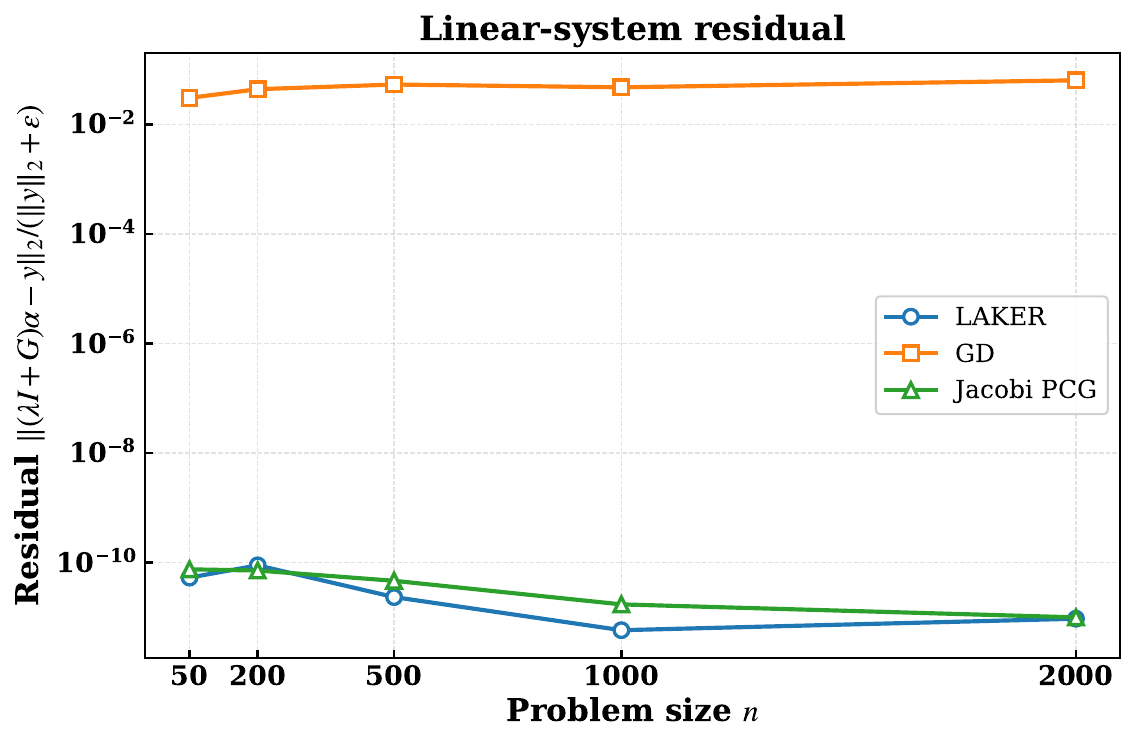}
}
\subfloat[\footnotesize Objective value gap.\label{fig:objective_gap_vs_n}]{\includegraphics[width=0.24\textwidth]{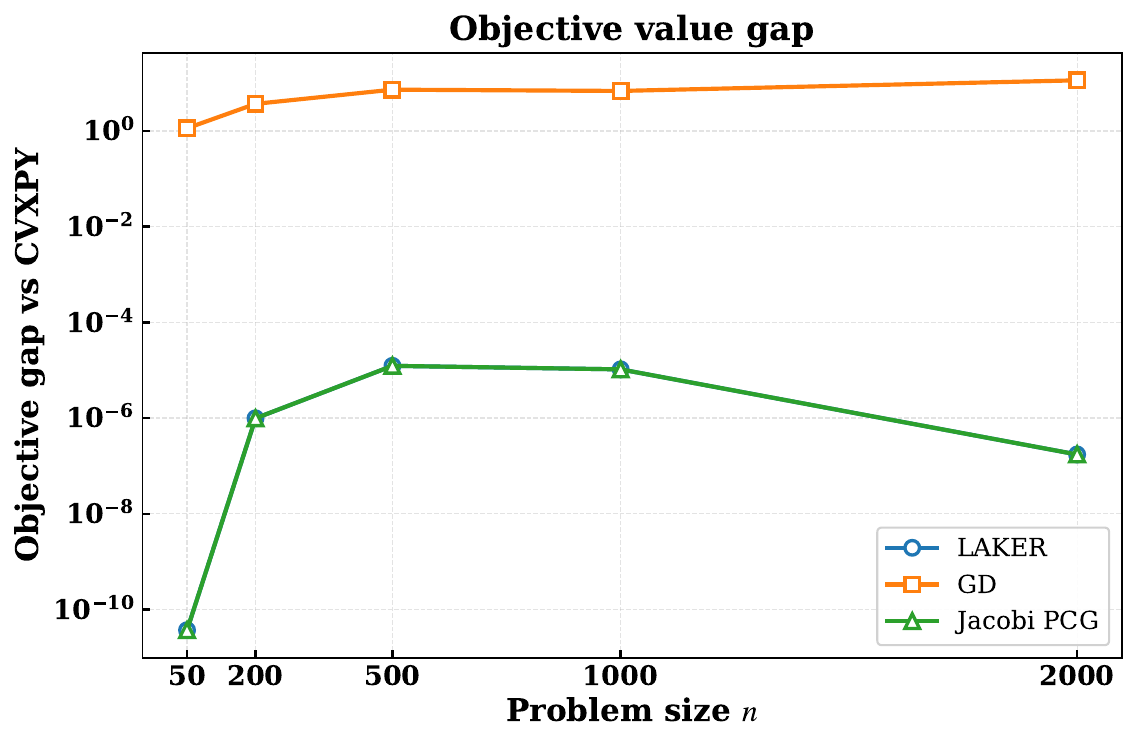}
}
\caption{Comparison between the algorithm LAKER, first-order method (GD), preconditioner method (Jocabi PCG) across different problem sizes, including condition number, iterations to target objective gap (cf. \eqref{eq:iterations}), and final solution accuracy in terms of residual (cf. \eqref{eq:residual}) and objective value gap (cf. \eqref{eq:objective-gap}).}
\label{fig:ours_vs_cvxpy_numerical_results}
\end{figure*}

\subsection{Evaluation Results}
This section reports evaluation results from two aspects. On the numerical side, we analyze condition number, convergence, and runtime efficiency, showing that the algorithm LAKER significantly improves spectral condition number and accelerates iterative solving. On the reconstruction side, we evaluate radio map reconstruction performance, showing that the algorithm LAKER achieves near-identical accuracy to the convex solver reference and slightly better large-scale performance than the Gaussian process regression-based method, while being significantly more efficient.

\vspace{1mm}
\subsubsection{Condition Number}
Figure ~\ref{fig:condition_number_comparison_vs_n} compares the condition numbers of the original system $\lambda I + G$ and the preconditioned systems across problem sizes. As $n$ increases from $50$ to $2000$, $\kappa(\lambda I + G)$ grows from $5.05\times10^3$ to $2.02\times10^5$, a roughly $40\times$ increase, confirming the severe spectral ill-conditioning of the exponential attention kernel. The Jacobi preconditioner provides negligible improvement: $\kappa(P_{\mathrm{J}}(\lambda I + G))$ remains essentially identical to the original system at all scales, since the diagonal Jacobi structure fails to capture the dominant off-diagonal interactions of the attention kernel. In contrast, the algorithm LAKER significantly reduces the condition number and keeps it nearly constant across all scales. Specifically, $\kappa(P(\lambda I + G))$ stays within the narrow range $[1.33\times10^2,\, 2.09\times10^2]$ for $n \in \{50,\dots,2000\}$. At $n=2000$, the algorithm LAKER reduces the condition number from $2.02\times10^5$ to $2.09\times10^2$, an improvement of nearly three orders of magnitude. This confirms that the learned preconditioner in the algorithm LAKER effectively reshapes the spectrum of the attention kernel system, which directly explains the faster convergence reported in the following experiments.

\vspace{1mm}
\subsubsection{Iterations}
Figure \ref{fig:iterations_to_accuracy_vs_n} evaluates the number of iterations required to reach a target objective gap of $10^{-3}$ with respect to the convex solver reference solution. The algorithm LAKER consistently requires significantly fewer iterations than Jacobi-preconditioned PCG across all problem sizes. Specifically, the algorithm LAKER converges in $16$, $21$, $25$, $28$, and $30$ iterations for $n=50, 200, 500, 1000,$ and $2000$, respectively, whereas Jacobi PCG requires $21$, $32$, $42$, $47$, and $59$ iterations. This corresponds to a reduction of approximately $20\%-50\%$ in iteration count. Moreover, the algorithm LAKER exhibits much slower growth in iteration complexity as the problem size increases, indicating improved efficiency. In contrast, the iteration count of Jacobi PCG grows more rapidly with $n$, reflecting its inability to effectively mitigate the severe ill-conditioning of the system. These results demonstrate that the learned preconditioner in the algorithm LAKER substantially improves the spectral properties of the system, leading to faster convergence to high-accuracy solutions.

\begin{figure}
    \centering
    \includegraphics[width=0.49\textwidth]{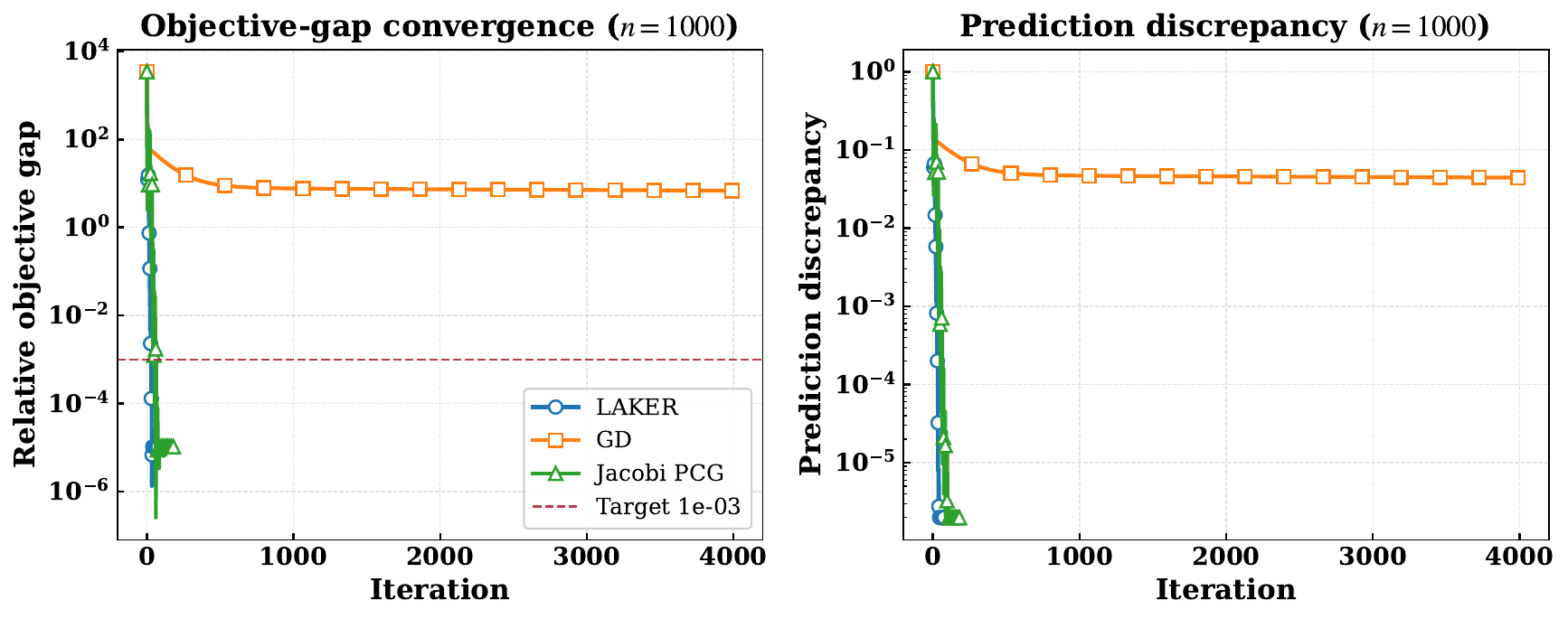}
    \caption{Convergence behavior for $n=1000$ including objective gap (cf. \eqref{eq:objective-gap}) and prediction discrepancy (cf. \eqref{eq:prediction-discrepancy}). The learned preconditioner in the algorithm LAKER leads to significantly faster time-to-accuracy and stable, scale-invariant convergence of the iterative solver.}
    \label{fig:representative_convergence_n1000}
\end{figure}

\vspace{1mm}
\subsubsection{Solution Accuracy}
Figures \ref{fig:linear_system_residual_vs_n} and \ref{fig:objective_gap_vs_n} compare the final linear-system residual and objective value gap achieved by different methods. The algorithm LAKER and Jacobi PCG both attain high-accuracy solutions across all problem sizes. The first-order method GD, by contrast, fails to converge under the iteration budget and yields residuals consistently near $10^{-2}$ and objective gaps exceeding $1$ at all problem sizes. From Figure ~\ref{fig:linear_system_residual_vs_n}, the relative residuals of the algorithm LAKER and Jacobi PCG remain at the level of $10^{-11}$--$10^{-12}$ for all tested scales. For example, at $n=1000$, the residuals are $5.81\times10^{-12}$ for the algorithm LAKER and $1.72\times10^{-11}$ for Jacobi PCG, while GD remains at $4.74\times10^{-2}$. A similar trend is observed at $n=2000$, where the algorithm LAKER and Jacobi PCG achieve residuals of $9.41\times10^{-12}$ and $9.97\times10^{-12}$, respectively, compared to $6.36\times10^{-2}$ for GD. Figure ~\ref{fig:objective_gap_vs_n} further shows that the algorithm LAKER and Jacobi PCG both maintain very small objective gaps relative to the convex solver, while GD remains far from the reference solution. Specifically, the algorithm LAKER achieves an objective gap of $9.90\times10^{-7}$ at $n=200$, $1.04\times10^{-5}$ at $n=1000$, and $1.74\times10^{-7}$ at $n=2000$, with Jacobi PCG giving nearly identical values. In contrast, GD yields objective gaps larger than $1$, reaching $6.83$ at $n=1000$ and $1.14\times10^{1}$ at $n=2000$. These results indicate that both PCG-based methods can recover solutions highly consistent with the convex solver reference, whereas first-order GD is inadequate for the severely ill-conditioned attention kernel systems considered here.

\begin{figure*}[t]
\centering
\includegraphics[width=0.96\textwidth]{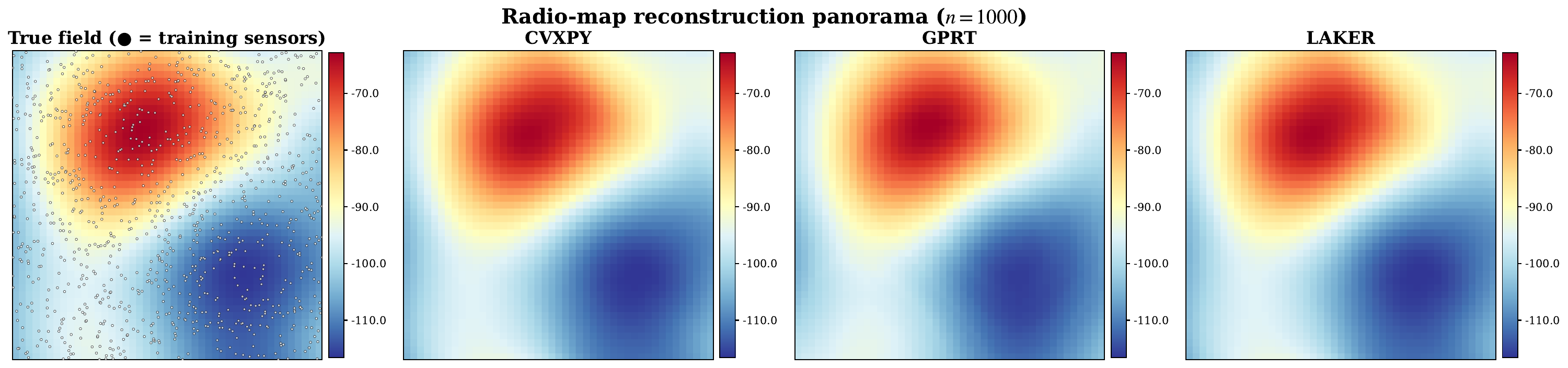}
\caption{Radio map reconstruction panorama for $n=1000$. From left to right: ground-truth field (with training sensor locations), the convex solver reference (CVXPY), the Gaussian process regression-based radio map reconstruction method (GPRT), and the algorithm LAKER. Algorithm LAKER produces spatial estimates that are visually indistinguishable from the the convex solver reference, accurately capturing both the dominant high-power region and the global propagation pattern. In contrast, the GPRT baseline exhibits noticeable deviations in both peak intensity and spatial smoothness.}
\label{fig:radio_map_row1_n1000}
\end{figure*}

\vspace{1mm}
\subsubsection{Convergence Behavior}
Figure ~\ref{fig:representative_convergence_n1000} illustrates the per-iteration convergence of the objective gap (left) and prediction discrepancy (right) at $n=1000$. The results clearly demonstrate that the algorithm LAKER achieves both the fastest convergence and the highest final accuracy among all methods.

Specifically, the objective gap of algorithm LAKER decreases from above $10^{0}$ to below the target threshold of $10^{-3}$ within $28$ iterations, and further converges to $1.04\times10^{-5}$ upon termination at $76$ iterations. The prediction discrepancy follows a similar trajectory, dropping from $\mathcal{O}(1)$ to below $10^{-5}$ within approximately $40$ iterations and remaining stable thereafter. Jacobi PCG also converges to a high-accuracy solution, but at a slower rate: it requires $47$ iterations to cross the target threshold and exhibits a more gradual decay in both metrics. This slower convergence is consistent with the condition number results, where the Jacobi preconditioner leaves $\kappa(P_{\mathrm{J}}(\lambda I + G))$ essentially unchanged from the original system. In contrast, the first-order method GD fails to converge meaningfully: the objective gap quickly stagnates near $6.83$ and the prediction discrepancy plateaus at approximately $4\times10^{-2}$, confirming that first-order methods are fundamentally inadequate for the severely ill-conditioned attention kernel system. The rapid and monotone convergence of algorithm LAKER, by contrast, is a direct consequence of the learned preconditioner compressing the spectrum of $(\lambda I + G)$ and enabling more uniform error reduction across spectral modes. Combined with the condition number results, these observations confirm that the algorithm LAKER significantly improves convergence efficiency while maintaining numerical stability across all tested scales.

\vspace{1mm}
\subsubsection{Spatial Reconstruction}
Figure ~\ref{fig:radio_map_row1_n1000} presents a qualitative comparison of radio map reconstruction at $n=1000$, showing the ground-truth field alongside the outputs of the convex solver, the Gaussian process regression-based radio map reconstruction method GPRT, and the algorithm LAKER. Both the convex solver and the algorithm LAKER faithfully recover the dominant spatial structure, including the high-power region in the upper left and the low-power region in the lower right, with consistent spatial structure throughout the domain. Visually, the two reconstructions are nearly indistinguishable, a finding supported by their RMSE values of $0.5240$ (LAKER) and $0.5242$ (the convex solver), differing by less than $2\times10^{-4}$. GPRT captures the overall spatial trend but exhibits noticeable discrepancies in peak intensity and produces an overly smooth field that misses finer spatial variations. This over-smoothing is a consequence of the fixed composite kernel in GPRT, whose built-in spatial prior imposes stronger regularity than the attention kernel. Quantitatively, GPRT achieves an RMSE of $0.6921$ at $n=1000$, approximately $32\%$ higher than LAKER. These results demonstrate that the algorithm LAKER matches the spatial fidelity of the convex solver reference solution while outperforming the model-driven GPRT baseline at this scale. We note that GPRT achieves lower RMSE at smaller problem sizes such as $n \le 500$, consistent with the well-known advantage of Gaussian process regression under limited observations; the advantage of the attention kernel formulation becomes apparent as $n$ grows.

\begin{figure}[t]
\centering
\includegraphics[width=0.48\textwidth]{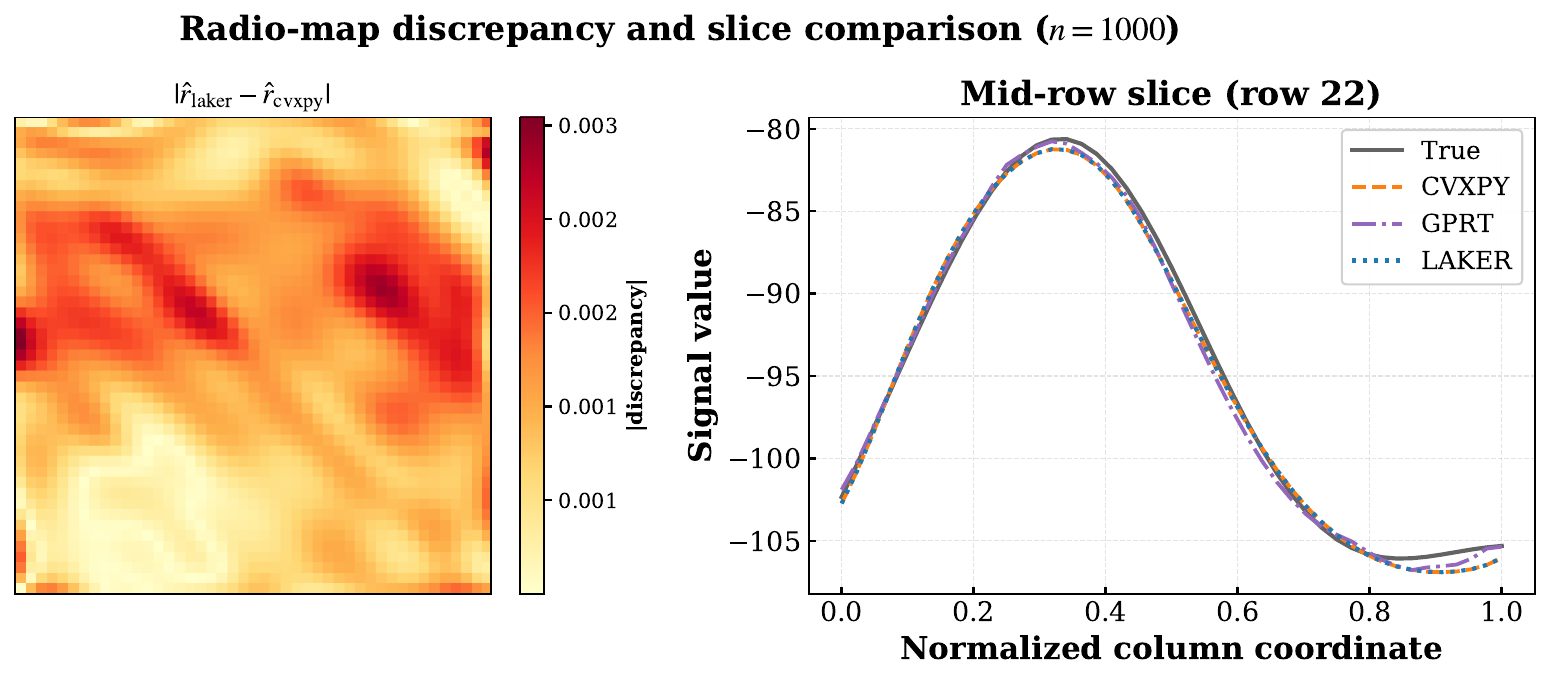}
\caption{Discrepancy and slice comparison for $n=1000$. Left: absolute difference between the algorithm LAKER and the convex solver reference solution. Right: mid-row slice comparing the ground truth, the convex solver, the Gaussian process regression-based method (GPRT), and the algorithm LAKER. The discrepancy remains very small across the domain, and LAKER closely aligns with both the convex solver and the ground truth, while GPRT exhibits visible deviations near peak and valley regions.}
\label{fig:radio_map_discrepancy_slice_n1000}
\end{figure}

\vspace{1mm}
\subsubsection{Discrepancy Analysis}
Figure ~\ref{fig:radio_map_discrepancy_slice_n1000} examines the reconstruction discrepancy at $n=1000$ from both a global (spatial map) and a local (cross-section) perspective. Specifically, the discrepancy map (left) shows the pointwise absolute difference $|\hat{r}_{\text{laker}} - \hat{r}_{\text{cvx}}|$ between the algorithm LAKER and the convex solver across the domain. The values remain uniformly below $3\times10^{-3}$ throughout, with no systematic bias, confirming that the algorithm LAKER closely reproduces the convex solver reference at a global level. The spatially distributed pattern of the discrepancy reflects local variations in the attention kernel regression solution rather than solver error, consistent with the near-zero residuals. The mid-row cross-section (right) provides a finer comparison of local reconstruction behavior. The curves of the ground truth, the convex solver, and the algorithm LAKER are nearly perfectly aligned across the entire cross-section, including both the peak region and the low-signal tail. GPRT, by contrast, underestimates the signal intensity near the peak and produces a smoother curve that deviates noticeably in the right portion of the domain (normalized coordinate $x > 0.7$). These observations confirm that LAKER faithfully captures both large-scale spatial trends and fine-grained local variations of the radio field, while GPRT's fixed kernel prior leads to over-smoothing at this problem scale.

\begin{table*}[!htbp]
\centering
\caption{Numerical performance across problem sizes ($\lambda=10^{-2}$, $\gamma=10^{-1}$). Bold indicates the best result in each column. The first-order method GD does not reach the target accuracy within the iteration budget.}
\label{tab:opt_results}
\small
\setlength{\tabcolsep}{4pt}
\renewcommand{\arraystretch}{1.2}
\begin{adjustbox}{max width=\textwidth}
\begin{tabular}{c ccc cccc cc cc}
\toprule
& \multicolumn{3}{c}{Numerical Accuracy}
& \multicolumn{4}{c}{Solver Time (seconds)}
& \multicolumn{2}{c}{Condition Number}
& \multicolumn{2}{c}{PCG Iterations} \\
\cmidrule(lr){2-4}
\cmidrule(lr){5-8}
\cmidrule(lr){9-10}
\cmidrule(lr){11-12}
$n$
& Obj.\ Gap
& Residual
& Pred.\ Disc.
& Convex Solver
& LAKER
& Jacobi
& GD
& $\kappa(\lambda I + G)$
& $\kappa(P(\lambda I + G))$
& LAKER
& Jacobi \\
\midrule
50
& \textbf{3.71e{-11}}
& \textbf{5.30e{-11}}
& \textbf{3.42e{-09}}
& 0.062
& \textbf{0.009}
& 0.011
& 0.204
& 5.05e{+03}
& \textbf{1.33e{+02}}
& \textbf{16}
& 21 \\

200
& \textbf{9.90e{-07}}
& \textbf{8.88e{-11}}
& \textbf{1.38e{-07}}
& 0.078
& \textbf{0.043}
& 0.071
& 3.324
& 2.02e{+04}
& \textbf{1.95e{+02}}
& \textbf{21}
& 32 \\

500
& \textbf{1.24e{-05}}
& \textbf{2.34e{-11}}
& \textbf{2.06e{-06}}
& 0.463
& \textbf{0.162}
& 0.198
& 7.499
& 5.04e{+04}
& \textbf{2.07e{+02}}
& \textbf{25}
& 42 \\

1000
& \textbf{1.04e{-05}}
& \textbf{5.81e{-12}}
& \textbf{1.98e{-06}}
& 2.875
& \textbf{0.411}
& 1.010
& 25.394
& 1.01e{+05}
& \textbf{1.79e{+02}}
& \textbf{28}
& 47 \\

2000
& \textbf{1.74e{-07}}
& \textbf{9.41e{-12}}
& \textbf{4.13e{-08}}
& 37.678
& \textbf{1.699}
& 19.348
& 102.657
& 2.02e{+05}
& \textbf{2.09e{+02}}
& \textbf{30}
& 59 \\
\bottomrule
\end{tabular}
\end{adjustbox}
\end{table*}

\begin{table}[!htbp]
\centering
\caption{Radio map reconstruction performance. Bold indicates the best result in each metric column.}
\label{tab:recon_results}
\small
\renewcommand{\arraystretch}{1.2}
\setlength{\tabcolsep}{3pt}
\begin{tabular*}{0.95\columnwidth}{@{\extracolsep{\fill}}c cc cc}
\toprule
& \multicolumn{2}{c}{RMSE}
& \multicolumn{2}{c}{NMSE} \\
\cmidrule(lr){2-3} \cmidrule(lr){4-5}
$n$ & LAKER & GPRT & LAKER & GPRT \\
\midrule
50
& 1.6946
& \textbf{1.3785}
& 3.31e{-04}
& \textbf{2.19e{-04}} \\

200
& 1.1610
& \textbf{0.6956}
& 1.53e{-04}
& \textbf{5.48e{-05}} \\

500
& 0.7841
& \textbf{0.7483}
& 6.63e{-05}
& \textbf{6.04e{-05}} \\

1000
& \textbf{0.5240}
& 0.6921
& \textbf{3.04e{-05}}
& 5.31e{-05} \\

2000
& \textbf{0.6212}
& 0.7585
& \textbf{4.87e{-05}}
& 7.26e{-05} \\
\bottomrule
\end{tabular*}
\end{table}

\vspace{1mm}
\subsubsection{Summary of Results}
Tables~\ref{tab:opt_results} and~\ref{tab:recon_results} summarize numerical layer and reconstruction layer performance across all problem sizes.

\vspace{1mm}
\textit{Numerical Performance.}
As shown in Table~\ref{tab:opt_results}, the algorithm LAKER achieves solutions that are numerically consistent with the convex solver reference across all scales. The objective gap ranges from $3.71\times10^{-11}$ to $1.74\times10^{-7}$, while the linear system residual and prediction discrepancy reach $10^{-12}$ and $10^{-9}$. In contrast, the first-order method GD fails to reach the target accuracy at all problem sizes, with objective gaps exceeding $1$, indicating its limitation for ill-conditioned attention kernel systems. From a spectral perspective, the condition number of $\lambda I + G$ increases from $5.05\times10^3$ to $2.02\times10^5$ as $n$ grows from $50$ to $2000$, whereas the preconditioned system remains stable in the range $1.33\times10^2$ to $2.09\times10^2$. This stabilization leads to nearly size-independent convergence. The algorithm LAKER requires only $16$ to $30$ iterations to reach the target accuracy, while Jacobi PCG requires $21$ to $59$ iterations and exhibits clear growth with problem size. In terms of solver time measured at the PCG tolerance, LAKER and Jacobi PCG solve the original problem using preconditioned PCG, while the convex solver and GD report their original runtimes. LAKER consistently outperforms all baselines. Compared with the convex solver, it achieves over $22\times$ speedup at $n=2000$, and the advantage increases with problem size. LAKER also outperforms Jacobi PCG, whose runtime grows rapidly with $n$, and GD, which remains inefficient due to slow convergence.

\vspace{1mm}
\textit{Reconstruction Performance.}
Table~\ref{tab:recon_results} compares the radio map reconstruction accuracy of algorithm LAKER and algorithm GPRT across different problem sizes. At small scales, GPRT achieves lower RMSE and NMSE, with clear advantages at $n=50$ and $n=200$. At $n=500$, the two methods exhibit nearly identical performance, with GPRT slightly outperforming LAKER in both metrics. As the problem size increases, LAKER becomes more competitive and eventually surpasses GPRT. At $n=1000$ and $n=2000$, LAKER achieves lower RMSE and NMSE, indicating improved reconstruction accuracy in larger-scale settings. In particular, the RMSE of LAKER decreases from $1.6946$ at $n=50$ to $0.5240$ at $n=1000$, and remains stable at larger scales, while GPRT exhibits higher error at $n \ge 1000$. A similar trend is observed for NMSE. Overall, the algorithm GPRT performs better in low-data regimes, while the algorithm LAKER demonstrates superior accuracy at moderate to large scales, highlighting its effectiveness for large-scale radio map reconstruction.

\section{Conclusion}
\label{sec:con}
This paper addressed the computational challenges of regularized attention kernel regression in spectrum cartography, where exponential attention kernels induce severely ill-conditioned linear systems whose condition numbers grow linearly with problem size, rendering standard iterative solvers ineffective. To overcome this bottleneck, we developed LAKER, a learning-based algorithm that acquires a data-dependent preconditioner via a difference-of-convex formulation and integrates it with a preconditioned conjugate gradient solver. The learned preconditioner directly targets the inverse spectral structure of the attention kernel system, reducing condition numbers by up to three orders of magnitude and enabling nearly size-independent preconditioned conjugate gradient convergence across all tested scales. Extensive experiments on spectrum-cartography-inspired radio map reconstruction demonstrated that LAKER achieves numerical accuracy matching the convex solver reference solution across all problem sizes, with negligible objective gaps and near-zero residuals, while attaining speedups of over twenty-fold compared to the convex solver at the largest tested scale. LAKER also outperforms the Gaussian process regression baseline in reconstruction accuracy at larger problem sizes while maintaining a more predictable computational cost. These results establish that learning-based spectral preconditioning provides an effective and scalable approach to attention kernel regression in wireless sensing applications. Future work includes extending the algorithm LAKER to online and dynamic spectrum cartography settings where measurements arrive sequentially.


\bibliographystyle{IEEEtran}
\bibliography{Kernel}

\end{document}